\documentclass{amsart}

\RequirePackage{amsthm,amsmath,amsfonts,amssymb}

\numberwithin{equation}{section}

\theoremstyle{plain}
\newtheorem{theorem}{Theorem}[section]
\newtheorem{proposition}[theorem]{Proposition}
\newtheorem{lemma}[theorem]{Lemma}
\newtheorem{corollary}[theorem]{Corollary}

\theoremstyle{definition}
\newtheorem{definition}[theorem]{Definition}
\newtheorem{example}[theorem]{Example}

\newtheorem{algorithm}[theorem]{Algorithm}

\theoremstyle{remark}
\newtheorem{remark}[theorem]{Remark}

\begin{document}


\title{$\Lambda$-linked coupling for drifting Brownian motions}



\author{Motoya Machida}


\address{Department of Mathematics\\
Tennessee Technological University\\
Cookeville, Tennessee 38505\\
USA}


\begin{abstract}
We raise a question on whether a dynamical system driven by Markov
process is Markovian, for which we are able to propose a criterion
and examples of positive case.
This investigation leads us to develop
(i)
a general construction of intertwining dual via Liggett duality,
and (ii)
a realization of $\Lambda$-linked coupling in a form of
dynamical system.
We show this construction of intertwining dual
and $\Lambda$-linked coupling
for an $n$-dimensional drifting Brownian motion
when it is a characteristic diffusion.
In particular, it includes an extension of Pitman's $2M-W$ theorem
by Rogers and Pitman as a special case.

\end{abstract}



\maketitle


\section{Introduction}
\label{intro}

This study was inspired by the recent development
for intertwining duals by
Fill and Lyzinski~\cite{FL} and Miclo~\cite{miclo}.
In this section we consider a diffusion operator
\begin{equation*}
  \mathcal{A}f
  = -\mu\frac{df}{dx}
  + \frac{1}{2}\frac{d^2f}{dx^2}
\end{equation*}
of one-dimensional Brownian motion
with constant drift $(-\mu)$,
and illustrate the connection between Liggett and intertwining
duality (Section~\ref{intro.ligg} and~\ref{intro.intwin}),
our notion of flow by Skorohod equations
(Section~\ref{intro.skd}),
and our construction of $\Lambda$-linked coupling
(Section~\ref{pitman.sec}).
In particular, we demonstrate how the realization of
$\Lambda$-linked coupling can be related back to
the work of Rogers and Pitman~\cite{PR}.

\subsection{Liggett dual}
\label{intro.ligg}

A process without drift term
is a Brownian motion, and denoted by $W(t)$.
Then a sample path
$\hat{X}(s)$ of drifting Brownian motion is constructed by
\begin{equation}\label{mu.drift}
  \hat{X}(s) = x - \mu s + W(s),
  \quad s \ge 0,
\end{equation}
starting from an initial state $\hat{X}(0)=x$;
in this paper
we use a caret-shaped symbol $\hat{X}(s)$
or a process $W(s)$ with time $s$
when we view them
as processes moving backward in time.
We set
\begin{equation}\label{bm.d}
  D^* = \{(z,y)\in\mathbb{R}^2: z<y\}
\end{equation}
as a state space dual to $D = \mathbb{R}$,
and construct a $D^*$-valued process
$X^*(t) = (Z(t),Y(t))$ by
\begin{equation}\label{bm.yz}
  \begin{cases}
    Y(t) = y + \mu t - W(t); \\
    Z(t) = z + \mu t + W(t),
  \end{cases}
\end{equation}
starting from $X^*(0) = (z,y)\in D^*$
until the absorbing time
$\zeta = \inf\{t\ge 0: Z(t)=Y(t)\}$
(and $X^*(t)$ remains at a coffin state for
$t\ge\zeta$; see III.3 of~\cite{rw1}).

\begin{theorem}\label{bm.cons}
Let $T > 0$, $x\in D$, and $x^*=(z,y)\in D^*$ be fixed.
Then $\hat{X}$ and $X^*$ satisfy
\begin{math}
  \mathbf{E}[\Gamma(x^*,\hat{X}(T))]
  = \mathbf{E}[\Gamma(X^*(T),x)]
\end{math}
with respect to the duality function
\begin{equation}\label{bm.gamma}
  \Gamma((z,y),x) = \begin{cases}
    1 & \mbox{ if $z< x\le y$; } \\
    0 & \mbox{ otherwise, }
  \end{cases}
\end{equation}
where
$\Gamma(X^*(T),x) = 0$ if $T\ge\zeta$.
\end{theorem}

\begin{proof}
Use a common Brownian motion $W(t)$,
and construct a sample path of $\hat{X}(s)$ and $X^*(t)$.
Set $M(T) = \max_{0\le v\le T}W(v)$,
and observe that $\zeta > T$ if and only if
$M(T) < \frac{y-z}{2}$.
The expectation
$\mathbf{E}[\Gamma(X^*(T),x)]$ can be expressed by
the probability
\begin{equation*}
  \begin{cases}
    \mathbb{P}\left(M(T)<\frac{y-z}{2},\, Y(T)\ge x\right)
    & \mbox{ if $x > \frac{z+y}{2}+\mu T$; } \\
    \mathbb{P}\left(M(T)<\frac{y-z}{2},\, Z(T) < x \right)
    & \mbox{ if $x \le\frac{z+y}{2}+\mu T$. }
  \end{cases}
\end{equation*}
Assuming that $x \le\frac{z+y}{2}+\mu T$,
we can apply the reflection principle of Brownian motion
(Corollary~I.13.3 of~\cite{rw1})
and reduce the above expression to
\begin{align*}
  & \mathbb{P}(W(T)< x-z-\mu T)
  - \mathbb{P}\left(M(T)\ge\frac{y-z}{2},\, W(T)<x-z-\mu T\right)
  \\ &
  = \mathbb{P}(W(T)> z-x+\mu T)
  - \mathbb{P}(W(T)> y-x+\mu T)
  \\ &
  = \mathbb{P}(z < \hat{X}(T) \le y)
  = \mathbf{E}[\Gamma(x^*,\hat{X}(T))] .
\end{align*}
The case for $x > \frac{z+y}{2}+\mu T$
is similarly completed.
\end{proof}

Liggett~\cite{liggett}
introduced the duality relation of Theorem~\ref{bm.cons},
notably including a notion of duality by Siegmund~\cite{siegmund},
and we call $X^*(t)$ of Theorem~\ref{bm.cons}
a \emph{Liggett dual} to $\hat{X}(s)$.

\subsection{Intertwining dual}
\label{intro.intwin}

The drifting Brownian motion~(\ref{mu.drift})
has the transition density function
$p(t,x,y) = (2\pi t)^{-1/2}\exp(-|y-x+\mu t|^2/2t)$,
and it is ``time-reversible''
with respect to an invariant function $\nu(x) = e^{-2\mu x}$
(i.e., \emph{$\nu$-symmetric}; see Section~\ref{diffusion}),
satisfying $\nu(x)p(t,x,y) = \nu(y)p(t,y,x)$.
In Section~\ref{diffusion} we briefly review a diffusion process
by means of differential operator, Markov semigroup, and stochastic
differential equation (SDE).
Unlike one-dimensional diffusions an invariant function
$\nu$ does not necessarily exist when a higher dimensional space is
considered.
In Definition~\ref{c.def}
we present a special case of characteristic diffusion
by which an $n$-dimensional drifting Brownian motion is designed to
achieve any invariant function of interest.

The Liggett dual~(\ref{bm.yz})
has the corresponding diffusion operator
\begin{align*}
  \mathcal{B}f
  & = \mu\left(
  \frac{\partial}{\partial y}
  + \frac{\partial}{\partial z}
  \right)f
  + \frac{1}{2}\left(
  \frac{\partial}{\partial y}
  - \frac{\partial}{\partial z}
  \right)^2\!\! f
\end{align*}
with boundary condition that $f(z,y)$ tends to zero as
$(z,y)$ approaches the boundary
$\partial = \{(y,y)\in\mathbb{R}^2:y\in\mathbb{R}\}$.
Then we find a harmonic function
\begin{math}
  h(z,y) = \int_z^y\nu(x)dx
\end{math}
on $D^*$,
for which $\mathcal{B}h = 0$ holds.
This enables us to construct an operator $\mathcal{B}^*$
by the Doob $h$-transform
\begin{align*}
  \mathcal{B}^*f
  & = \frac{1}{h}\mathcal{B}[hf]
  \\
  & = \left[
    \mu + 2\mu\coth(\mu(y-z))
    \right]
  \frac{\partial f}{\partial y}
  + \left[
    \mu - 2\mu\coth(\mu(y-z))
    \right]
  \frac{\partial f}{\partial z}
  \\
  & \hspace{0.55in}
  + \frac{1}{2}\left(
  \frac{\partial}{\partial y}
  - \frac{\partial}{\partial z}
  \right)^2\!\! f
\end{align*}

We define a Markov kernel density $\lambda$
from $D^*$ to $D$ by
\begin{equation*}
  \lambda((z,y),x) = \frac{1}{h(z,y)}\Gamma((z,y),x)\nu(x)
\end{equation*}
and introduce the corresponding Markov kernel $\Lambda$ by
\begin{equation*}
  \Lambda f(z,y) = \frac{1}{h(z,y)}\int_z^y \nu(x)f(x)dx
\end{equation*}
for any bounded measurable function $f$ on $D$.
Then $\mathcal{B}^*$ satisfies
\begin{math}
  \Lambda\mathcal{A}f = \mathcal{B}^*\Lambda f ,
\end{math}
and it is called an \emph{intertwining dual} to $\mathcal{A}$.
In a setting of Markov chains
Diaconis and Fill~\cite{DF} observed that an intertwining dual
can be viewed as a Doob $h$-transform of the Siegmund dual of
the time-reversed Markov chain,
and Fill and Lyzinski~\cite{FL} demonstrated
the analogous result for diffusions on $[0,1]$.
The above construction of intertwining dual
coincides with the one obtained by Miclo~\cite{miclo}.

In general the intertwining duality can be introduced between two
Markov semigroups $P_t$ and $Q^*_t$, namely by
$\Lambda P_t = Q^*_t\Lambda$.
In Chapter~\ref{duality} we present intertwining duality
in terms of Markov semigroups, and discuss
a general construction of intertwining dual.
Once a Liggett dual $Q_t$ is constructed, the Doob $h$-transform
$Q^*_t$ is an intertwining dual to $P_t$;
see Proposition~\ref{twin.prop}.

\subsection{Skorohod equations and flow}
\label{intro.skd}

Let $y\in\mathbb{R}$ and $T > 0$ be fixed.
Provided a sample path $X = (X(t))_{0\le t\le T}$,
we can impute a Brownian motion $\omega(t)$ by
\begin{equation*}
  \omega(t) = X(t) - X(0) - \mu t,
  \quad 0\le t\le T,
\end{equation*}
as if $X(t)$ were governed by
\begin{math}
  X(t) = X(0) + \mu t + \omega(t) .
\end{math}
If $y \ge X(0)$, we can set
a nondecreasing process
\begin{equation*}
  L(t) = -\min_{0\le v\le t}[y-X(0)-2\omega(v)]\wedge 0
\end{equation*}
starting from $L(0)=0$.
This process $L(t)$
is uniquely determined as a solution $Y(t)$ and $L(t)$
to the following equations of \emph{Skorohod type}
\begin{equation}
  \begin{cases}\label{skd.n1}
    Y(t) = y + \mu t - \omega(t) + L(t); \\ 
    L(t) = \displaystyle\int_0^t I_{\{X(v)=Y(v)\}}\,dL(v),
  \end{cases}
\end{equation}
which was first proposed by
Saisho and Tanemura~\cite{saisho}.
The solution $Y(t)$ becomes an upper bound for $X(t)$,
and maintains
\begin{math}
  Y(t)-X(t) = y-X(0)-2\omega(t)+L(t)
\end{math}
for $0\le t\le T$.
Then we can construct a \emph{flow}
\begin{equation}\label{omega.n1}
  \tilde{\Theta}_{y,T}(X) = \begin{cases}
    (-\omega(t))_{0\le t\le T} & \mbox{ if $X(0) > y$; } \\
    (\omega(t) - L(t))_{0\le t\le T} & \mbox{ if $y \ge X(0)$, }
  \end{cases}
\end{equation}
which maps from a sample path $(X(t))_{0\le t\le T}$ to a sample path
$\tilde{\Theta}_{y,T}(X)$ on the interval $[0,T]$.

\begin{proposition}\label{pitman.cons}
Construct $(\hat{X}(s))_{0\le s\le T}$ by~(\ref{mu.drift}) using
a Brownian motion $W(t)$,
and set the backward sample path
$\hat{X}(T-\cdot) = (\hat{X}(T-t))_{0\le t\le T}$.
Then
$\tilde{\Theta}_{y,T}(\hat{X}(T-\cdot))(T)$ is distributed as $W(T)$.
\end{proposition}

\begin{proof}
We set $X'(t) = x - \mu T + W(T-t)$
and $Y'(T) = y - \tilde{\Theta}_{y,T}(\hat{X}(T-\cdot))(T)$,
and complete the proof by claiming that
$Y'(T)$ is distributed as $y + W(T)$.
Observe that
$\omega(t) = W(T-t) - W(T)$ by imputation,
and that $\{\hat{X}(T) > y\} = \{X'(0) > y\} = \{Y'(T)<x-\mu T\}$
by~(\ref{omega.n1}).
For $u\ge 0$ we can immediately obtain
\begin{equation*}
  \mathbb{P}(Y'(T)<x-\mu T-u)
  = \mathbb{P}(y+W(T)<x-\mu T-u)
\end{equation*}
For $u\ge x$ we can observe that
\begin{align*}
  & \mathbb{P}(Y'(T)\ge 2u-x-\mu T)
  = \mathbb{P}\left(
  \max_{0\le v\le T}X'(v)\ge u-\mu T,\,
  X'(0)\le y
  \right)
  \\ & \hspace{0.25in}
  + \mathbb{P}\left(
  \max_{0\le v\le T}X'(v)< u-\mu T,\,
  2(u-\mu T)-y\le X'(0)\le y
  \right)
\end{align*}
If $y\ge u-\mu T$ then
by setting $M(T) = \max_{0\le v\le T}W(v)$
we can apply the reflection principle
(as in the proof of Theorem~\ref{bm.cons})
and reduce the above expression to
\begin{align*}
  & \mathbb{P}(2(u-\mu T)-y\le x-\mu T + W(T)\le y)
  \\ & \hspace{0.25in}
  + \mathbb{P}\left(
  x+M(T)\ge u,\,
  x-\mu T+W(T)< 2(u-\mu T)-y
  \right)
  \\ &
  = \mathbb{P}(y+W(T)\ge 2u-x-\mu T).
\end{align*}
The case for $y<u-\mu T$ is similarly argued.
\end{proof}

\subsection{$\Lambda$-linked coupling}
\label{pitman.sec}

The notion of $\Lambda$-linked coupling
was originally proposed by Diaconis and Fill~\cite{DF}
in the setting of Markov chains.
In Section~\ref{duality} we propose a general construction
for the desirable properties of $\Lambda$-linked coupling
when characteristic diffusions of Definition~\ref{c.def} are
considered.

In a construction of $X^*(T)$ by (\ref{bm.yz})
we use the flow $\tilde{\Theta}_{y,T}(X)(T)$ of (\ref{omega.n1})
in the place of $W(T)$.
It defines a map $\Psi^*_T((z,y),X)$
from a sample path $(X(t))_{0\le t\le T}$ to $(Z(T),Y(T))$ by
\begin{equation}\label{bm.yz2}
  \begin{cases}
    Y(T) = y + \mu T - \tilde{\Theta}_{y,T}(X)(T); \\
    Z(T) = z + \mu T + \tilde{\Theta}_{y,T}(X)(T) .
  \end{cases}
\end{equation}
until the absorbing time
$\zeta = \inf\{t\ge 0: Z(t)=Y(t)\}$.

\begin{theorem}\label{bm.cons2}
Let $T> 0$ and $x^*\in D^*$ be fixed.
Then
(a) provided any sample path $X = (X(t))_{0\le t\le T}$,
$X^*(T) = \Psi^*_T(x^*,X)$ satisfies
\begin{math}
  \Gamma(x^*,X(0))=\Gamma(X^*(T),X(T)),
\end{math}
and (b) provided the backward sample path
$\hat{X}(T-\cdot) = (\hat{X}(T-t))_{0\le t\le T}$
of Proposition~\ref{pitman.cons},
$\Psi^*_T(x^*,\hat{X}(T-\cdot))$
is distributed as $X^*(T)$ of~(\ref{bm.yz}).
\end{theorem}

\begin{proof}
(a)
If $X(0)\le z$ or $X(0) > y$
then $\Gamma(X^*(T),X(T)) = 0$.
If $z < X(0) \le y$
then $(Y(t))_{0\le t\le T}$ is a solution to (\ref{skd.n1}),
by which we can easily verify
$\Gamma(X^*(T),X(T)) = 1$.
(b) is an immediate consequence of
Proposition~\ref{pitman.cons}.
\end{proof}

A relationship with Liggett dual
can be observed when we set $X^*(T) = \Psi^*_T(x^*,\hat{X}(T-\cdot))$
as in Theorem~\ref{bm.cons2}(b).
Since $\Gamma(x^*,\hat{X}(T)) = \Gamma(X^*(T),\hat{X}(0))$
by Theorem~\ref{bm.cons2}(a),
it provides an alternative proof for Theorem~\ref{bm.cons}.
Furthermore, a remarkable connection to intertwining dual
can be established in a construction of $X^*(t)$ by (\ref{bm.yz2}).
Here we sample $X(0)$ randomly from $\lambda(x^*,\cdot)$,
and generate a Markov process $X(t)$ by
\begin{equation*}
  X(t) = X(0) - \mu t + W(t),
  \quad t\ge 0.
\end{equation*}
The resulting bivariate process $(X^*(t),X(t))$ is Markovian,
and it becomes a $\Lambda$-linked coupling of Theorem~\ref{link.thm};
see Proposition~\ref{s.twin.prop}.
By Theorem~\ref{link.thm} the Markov process $X^*(t) = (Z(t),Y(t))$
is governed by the
intertwining dual operator $\mathcal{B}^*$.

The bivariate process
$U(t) = [Y(t)+Z(t)]/2$ and $V(t) = [Y(t)-Z(t)]/2$
has the diffusion operator
\begin{equation*}
  \mathcal{B}^*f
  = \mu\frac{\partial f}{\partial u}
  + 2\mu\coth(2\mu v)\frac{\partial f}{\partial v}
  + \frac{1}{2}\frac{\partial^2 f}{\partial v^2} ;
\end{equation*}
in particular, $V(t)$ becomes a Bessel process
\begin{equation*}
    dV(t) = 2\mu\coth(2\mu V(t))\,dt + dW(t)
\end{equation*}
for the drifting Brownian motion $W(t)-2\mu t$.
It can start from $V(0)=0$, and never hits $0$ again;
see further discussion in Section~\ref{e.link}.
Thus, starting from $(X^*(0),X(0)) = ((0,0),0)$,
the coupling satisfies
\begin{math}
  Y(t)-X(t) = 4\mu t - 2W(t) + 2M(t)
\end{math}
with
\begin{equation*}
  2M(t) = 2\max_{0\le v\le t}\left[
    W(v) - 2\mu v
    \right] .
\end{equation*}
Therefore, it provides a construction of $V(t)$ by
\begin{equation*}
  V(t) = 2M(t) - [W(t) - 2\mu t] .
\end{equation*}
This sample path construction
was obtained by Pitman~\cite{pitman} for $\mu=0$,
and extended by Rogers and Pitman~\cite{PR}.
Their results are collectively called
Pitman-type $2M-W$ theorems, and were
extensively studied by Matsumoto and Yor~\cite{matsumoto1}
and many others cited therein.

Intertwining duality has been studied in relation with
the question of when a function $\phi(X(t))$
of a Markov process $X(t)$ is Markovian;
see \cite{pal} for a brief review of the literature.
General criteria such as Theorem~2 of~\cite{PR} for the Markovian
question were used for
the sample path construction of $V(t)$.
In the present paper we raise the question of when a ``random''
dynamical system $\Psi^*_t(x^*,X)$ is Markovian, provided that $X$
is a Markov process.
As demonstrated in this section,
our criterion (Proposition~\ref{s.twin.prop}) for this new
question on $\Psi^*_t(x^*,X)$ can be successfully applied
to the analysis of diffusion process $V(t)$.

The $\Lambda$-linked coupling of Theorem~\ref{link.thm}
also implies that
the regular conditional probability distribution
$\mathbb{P}(X(t)\in\cdot|X^*(t)=(z,y))$
has a probability density function (pdf) on $(z,y]$ proportional to
the invariant function $\nu(x) = e^{-2\mu x}$.
The equivalent observation was made by~\cite{PR}
that $\mathbb{P}([B(t) - 2\mu t]\in\cdot|V(t)=y)$
has a pdf on $(-y,y]$ proportional to $e^{-2\mu x}$.

\subsection{Outline for the rest}

In Section~\ref{euler.app}
we begin our investigation
of an $n$-dimensional drifting Brownian motion $X(t)$
and its time-reversed $\hat{X}(s)$
with Euler approximations.
Euler schemes and other forms of algorithm in approximation are
necessary ingredient in describing a general construction of
stochastic processes when they are elaborately coupled.
In Section~\ref{hypo.sec} we propose a stochastic process
$\partial Y^*(t)$
of hypographical surface as an upper bound for $X(t)$,
and present a coupled construction with time-reversed $\hat{X}(s)$
in Algorithm~\ref{y.cons.alg}.
In Section~\ref{rel.skd}
we examine it with equations of Skorohod type,
which leads us to a construction of Liggett dual in
Proposition~\ref{ligg.cons}.
A coupled construction of $X(t)$ and $\partial Y^*(t)$
forward in time (Algorithm~\ref{s.ligg.alg})
enables us to define
an $n$-dimensional version of flow.
In Section~\ref{impute.sec}
we present
a stronger version of Proposition~\ref{pitman.cons}
(namely Proposition~\ref{y.hat.dist}), claiming
that the flow~(\ref{omega.n1}) is distributed
as the Wiener measure on the interval $[0,T]$.
Together we are able to construct a $\Lambda$-linked coupling in
Proposition~\ref{s.ligg.cons}.

Kent~\cite{kent} observed that a drifting Brownian motion can be
designed for arbitrary invariant function $\nu(x)$ on $\mathbb{R}^n$
(characteristic diffusions; see Definition~\ref{c.def}).
In Section~\ref{m1.sec} we present three examples of characteristic
diffusion, and use them to illustrate a construction of hypographic
surface.
In Section~\ref{ligg.cons.sec}
we continue our exploration for examples of Liggett dual.
The exploration culminates in Section~\ref{c.link} for 
the construction of intertwining dual
for each case of the examples.
In particular, we consider
a posterior density $\nu_H(x)$ on a hyperplane $H$
in Example~\ref{b.model},
and look at a possibility of Monte Carlo simulation
out of Example~\ref{b.model.ligg} and~\ref{b.model.ent}.

\section{Characteristic diffusions}
\label{diffusion}

We introduce a diffusion operator $\mathcal{A}$ on
$\mathbb{R}^n$ by
\begin{equation}\label{c.ope}
  \mathcal{A}_x{f}
  = \sum_{i=1}^n
  \left[\frac{1}{2}
    \frac{\partial^2}{\partial x_i^2}{f}(x)
    - \beta_i(x) \frac{\partial}{\partial x_i}{f}(x)
    \right] ,
\end{equation}
where the subscript on $\mathcal{A}$
indicates variables to differentiate.
By
\begin{equation*}
  \mathcal{A}^\dagger_y{f}
  = \sum_{i=1}^n
  \left[\frac{1}{2}
    \frac{\partial^2}{\partial y_i^2}{f}(y)
    + \frac{\partial}{\partial y_i}
    (\beta_i(y){f}(y))
    \right]
\end{equation*}
we denote the adjoint operator of $\mathcal{A}$.
In what follows we assume that
the drift coefficients $\beta_i(x)$'s are smooth enough
(differentiability and H\"{o}lder continuity for their
derivatives) so that 
a fundamental solution exists;
see \cite{friedman1,ito53,skorokhod}
for sufficient conditions for existence and uniqueness.
Thus, the differential operator $\mathcal{A}$ uniquely determines
a positive and conservative
[i.e., $\int p(t,x,y)\,dy = 1$] transition density function
$p(t,x,y)$.
It satisfies the parabolic equations
\begin{align*}
  & \frac{\partial}{\partial t}p(t,x,y)
  = \mathcal{A}_x p(t,x,y);
  \\
  & \frac{\partial}{\partial t}p(t,x,y)
  = \mathcal{A}^\dagger_y p(t,x,y),
\end{align*}
which are respectively referred as Kolmogorov backward and forward
equation.


Let $\mathbb{R}_+$ be the half line~$[0,\infty)$,
and let $C(\mathbb{R}_+,\mathbb{R}^n)$ be
the space of all continuous functions from
$\mathbb{R}_+$ to $\mathbb{R}^n$.
In terms of SDE the distribution determined by
(\ref{c.ope}) corresponds to a solution to
\begin{equation}\label{c.sde}
  dX(t) = -\beta(X(t))dt + dW(t),
\end{equation}
where
$\beta(x) = [\beta_1(x),\ldots,\beta_n(x)]^T$
is the column vector of drift coefficients
and $W(t)$ is an $n$-dimensional Brownian motion.
That is, the solution $X(t)$ starting at $X(0) = x$
corresponds to the probability measure $\mathbb{P}_x$
on $C(\mathbb{R}_+,\mathbb{R}^n)$ which satisfies
\begin{equation}\label{pp.x}
  \mathbb{P}_x(X(t_i)\in dx_i,i=1,\ldots,N)
  = \prod_{i=1}^N p(t_i-t_{i-1},x_{i-1},x_i)\,dx_i
\end{equation}
with $x_0 = x$.
Here the event ``$X(t_i)\in dx_i,i=1,\ldots,N$''
is the measurable set
$\{X\in C(\mathbb{R}_+,\mathbb{R}^n): X(t_i)\in dx_i,i=1,\ldots,N\}$,
and $X$ is identified with an element of
$C(\mathbb{R}_+,\mathbb{R}^n)$.
By $\mathbf{E}_{\mathbb{P}_x}[F(X)]$
we denote the expectation with respect to the probability measure
$\mathbb{P}_x$
for any measurable function $F$ on $C(\mathbb{R}_+,\mathbb{R}^n)$.

A strictly positive function $\nu$ on $\mathbb{R}^n$ is called
\emph{invariant} if it satisfies
\begin{equation*}
  \nu(y) = \int \nu(x) p(t,x,y)\,dx
\end{equation*}
for any $t > 0$ and $y\in\mathbb{R}^n$.
Then
\begin{equation*}
  \tilde{p}(t,x,y) = \frac{\nu(y)}{\nu(x)}p(t,y,x)
\end{equation*}
is the time-reversed transition density with respect to $\nu$,
and it satisfies
\begin{equation*}
  \frac{\partial}{\partial t}\tilde{p}(t,x,y)
  = \frac{1}{\nu(x)} \mathcal{A}^\dagger_x[\nu(x)\tilde{p}(t,x,y)].
\end{equation*}
The transition density function $p$ is called
\emph{$\nu$-symmetric} if $p = \tilde{p}$.
Kent showed (in Section~4 of~\cite{kent}) that
$p$ is $\nu$-symmetric if and only if
the operators $\mathcal{A}$ and $\mathcal{A}^\dagger$ satisfy
$\mathcal{A}_x{f}(x) = 
\frac{1}{\nu(x)} \mathcal{A}^\dagger_x[\nu(x){f}(x)]$,
which is equivalently characterized by
$\frac{1}{2}\frac{\partial}{\partial x_i} \nu
= -\beta_i \nu$
for $i=1,\ldots,n$.

\begin{definition}\label{c.def}
Let $\gamma$ be a real-valued function on $\mathbb{R}^n$,
and let $\nu(x) = \exp(-2\gamma(x))$.
We call (\ref{c.ope})
a \emph{characteristic diffusion} for $\nu$
if the drift coefficient $\beta_i$ satisfies
$\beta_i = \frac{\partial}{\partial x_i}\gamma$
for each $i=1,\ldots,n$.
\end{definition}

In Definition~\ref{c.def}
the scalar function $-\gamma(x)$ is regarded as a potential energy,
and $\beta(x)$ is the gradient $\nabla\gamma(x)$.
Then $\nu$ is invariant, and
the transition density $p$ is $\nu$-symmetric
(Section~4 of~\cite{kent}).

\begin{lemma}\label{rev.lem}
Let $\mathbb{P}_x$ be the probability measure determined by
a characteristic diffusion of Definition~\ref{c.def},
and let $T > 0$ be fixed.
Assuming that a function $F_T$ on $C([0,T],\mathbb{R}^n)$
is integrable in either side of~(\ref{rev.int}), we have
\begin{equation}\label{rev.int}
  \int \nu(x)\mathbf{E}_{\mathbb{P}_x}[F_T(X)]dx
  = \int \nu(y)\mathbf{E}_{\mathbb{P}_y}[F_T(X(T-\cdot)]dy
\end{equation}
\end{lemma}

\begin{proof}
It suffices to show (\ref{rev.int}) for
$F_T(X) = \prod_{i=0}^N I_{E_i}(X(t_i))$
with Borel subsets $E_i$'s of $\mathbb{R}$
and $0=t_0<\cdots<t_N=T$
(cf. Section~II-38 of~\cite{rw1}),
where $I_{E_i}(x)$ denotes the indicator function on $E_i$.
Then the left-hand side of (\ref{rev.int}) can be expressed as
\begin{equation*}
  \int_{E_0}\!\!\cdots\int_{E_N}
  \nu(x_0)dx_0
  \prod_{i=1}^N p(t_i-t_{i-1},x_{i-1},x_i)\,dx_i
\end{equation*}
and the right-hand side becomes
\begin{equation*}
  \int_{E_N}\!\!\!\cdots\int_{E_0}
  \nu(x_N)dx_N
  \prod_{i=N}^1 p(t_i-t_{i-1},x_i,x_{i-1})\,dx_{i-1}
\end{equation*}
By repeatedly applying the $\nu$-symmetry of $p$
we can verify that they are equal.
\end{proof}

In general a diffusion operator $\mathcal{B}$ is accompanied
with domain $\mathcal{D}_{\mathcal{B}}$,
and it uniquely determines a sub-Markov semigroup $Q_t$.
When $Q_t$ is conservative [i.e., $\int Q_t(x,dy) = 1$],
we can correspond it to a probability measure $\mathbb{Q}_x$
in the same way we have constructed $\mathbb{P}_x$ satisfying
(\ref{pp.x}).
Furthermore,
it characterizes a weak solution of SDE
as $\mathbb{Q}_x$ represents
a solution to martingale problem satisfying
the Dynkin's formula
\begin{equation*}
  \mathbf{E}_{\mathbb{Q}_x}[f(X(t))] - f(x)
  = \mathbf{E}_{\mathbb{Q}_x}\left[
    \int_0^t\mathcal{B}f(X(v))dv
    \right] ,
\end{equation*}
which corresponds to the analytical relationship between
$\mathcal{B}$ and $Q_t$ for $f\in\mathcal{D}_{\mathcal{B}}$;
see Section~V.20 of~\cite{rw1}.

\section{$\Lambda$-linked coupling}
\label{duality}

In the rest of this paper
we set $D = \mathbb{R}^n$,
and consider a semigroup $P_t$ on $D$
for characteristic diffusion (Definition~\ref{c.def}).
We introduce another Polish space $D^*$
as a ``dual'' state space.
It is assumed that $D^*$ is open relative to
its extension $\bar{D}^* = D^*\cup\partial$,
and that $\bar{D}^*$ is Polish
with different choice of metric.

\begin{definition}\label{mds}
Let $(\psi_t)_{t\ge 0}$ be
a family of measurable maps $\psi_t$ from
$\bar{D}^*\times C([0,t],\mathbb{R}^n)$ to $\bar{D}^*$,
and let
\begin{math}
  \zeta(x,\omega) = \inf\{t\ge 0: \psi_t(x,\omega)\in\partial\}
\end{math}
be a map from
$D^*\times C(\mathbb{R}_+,\mathbb{R}^n)$ to
$\mathbb{R}_+\cup\{\infty\}$.
Then $\psi_t$ is said to be
a \emph{dynamical system} driven by a Markovian ``noise''
$\omega\in C(\mathbb{R}_+,\mathbb{R}^n)$
according to some probability measure on
$C(\mathbb{R}_+,\mathbb{R}^n)$
if for each
\begin{math}
  (x,\omega)\in
  D^*\times C(\mathbb{R}_+,\mathbb{R}^n)
\end{math}
(a)
$\psi_0(x,\omega) = x$,
and (b)
$\psi_t(x,\omega)$ is continuous on $t\in [0,\zeta(x,\omega))$.
Furthermore, it is called a \emph{Markov dynamical system}
if it also satisfies
$\psi_t(x,\omega) = \psi_{t-s}(\psi_s(x,\omega),\omega(\cdot+s))$,
$0\le s< t< \zeta(x,\omega)$,
for each
\begin{math}
  (x,\omega)\in
  D^*\times C(\mathbb{R}_+,\mathbb{R}^n) .
\end{math}
\end{definition}

A Markov dynamical system $\psi_t$ is usually driven
by the Wiener measure $\mathbb{W}$,
and the corresponding Markov process (until terminated)
is determined by the sub-Markov semigroup
\begin{math}
  Q_tf(x) = \mathbf{E}_{\mathbb{W}}\left[
    f(\psi_t(x,\omega))I_{\{t<\zeta(x,\omega)\}}
    \right] ,
\end{math}
where $I_{\{t<\zeta(x,\omega)\}}$ is the indicator function
of $\{t<\zeta(x,\omega)\}$.
Using a Brownian motion $W$ and an initial value $X(0) = x$,
a Markov process $X(t)$
can be expressed by $X(t) = \psi_t(x,W)$
until the absorbing time $\zeta(x,W)$.

\subsection{$\Lambda$-linked coupling}
\label{link.sec}

A Markov kernel density $\lambda(x^*,x)$
from $D^*$ to $D$ is called a \emph{link}.
In particular, $\lambda(x^*,\cdot)$ is a probability
density on $D$ [i.e., $\int\lambda(x^*,x) dx = 1$].

\begin{definition}\label{link}
We assume that
$E = \{(x^*,x)\in D^*\times D: \lambda(x^*,x)>0\}$ is Polish.
Let $Q^*_t$ be a Markov semigroup on $D^*$,
and let $V_t$ be a Markov semigroup on $E$.
Then $V_t$ is said to be
\emph{$\Lambda$-linked} between $P_t$ and $Q^*_t$ if
(a)
$V_tf = P_tf$ for $f(x^*,x) = f(x)$ on $E$,
and (b)
\begin{equation}
  \label{v.link}
  \int \lambda(x^*,x) V_tg(x^*,x)\,dx
  = \int Q^*_t(x^*,dy^*)\int\lambda(y^*,y)g(y^*,y)\,dy
\end{equation}
for any $x^*\in D^*$
and for any bounded measurable function $g$ on $E$.
\end{definition}

By $\Lambda$ we denote the map
\begin{equation*}
  \Lambda[f](x^*) = \int\lambda(x^*,x)f(x)\,dx
\end{equation*}
from bounded measurable functions $f$ on $D$
to bounded measurable functions $\Lambda[f]$
on $D^*$.
Then the semigroup $V_t$ of Definition~\ref{link}
implies that
\begin{align*}
\Lambda[P_t f]
& = \int \lambda(x^*,x)V_tf(x)\,dx
\\
& = \int Q^*_t(x^*,dy^*)\int\lambda(y^*,y)f(y)\,dy
= Q^*_t[\Lambda f] ;
\end{align*}
thus, $P_t$ and $Q^*_t$ are ``$\Lambda$-linked.''
It is common to call
$Q^*_t$ an \emph{intertwining dual} of $P_t$
with respect to $\Lambda$
when $\Lambda P_t = Q^*_t\Lambda$ holds.
The corresponding infinitesimal operator $\mathcal{B}^*$
of $Q^*_t$, if it exists, is an intertwining dual of $\mathcal{A}$
if $\Lambda\mathcal{A} = \mathcal{B}^*\Lambda$.

When a Markov process $(X^*(t),X(t))$ is generated by the Markov
semigroup $V_t$ of Definition~\ref{link}, the marginal distribution
of $X^*(t)$ may not be Markovian.
However, we obtain the result similar to Theorem~2 of~\cite{PR}.

\begin{theorem}\label{link.thm}
Fix $x^*_0\in D^*$, and construct a
probability measure $\mathbb{V}_{\lambda(x^*,\cdot)}$
on $C(\mathbb{R}_+,E)$ satisfying for $0=t_0\le t_1<\cdots<t_N$,
\begin{align*}
  & \mathbb{V}_{\lambda(x^*_0,\cdot)}
  ((X^*(t_i),X(t_i))\in dx_i^*\times dx_i,i=1,\ldots,N)
  \\
  & = \int\lambda(x^*,x_0)\,dx_0\times
  \prod_{i=1}^N V_{t_i-t_{i-1}}((x_{i-1}^*,x_{i-1}),dx_i^*\times dx_i) .
\end{align*}
Then $X^*(t)$ is Markovian with initial state $X^*(0) = x^*_0$,
and governed by the Markov semigroup $Q^*_t$.
\end{theorem}

\begin{proof}
Similarly to the proof of Theorem~2 of~\cite{PR}
we can verify for $0=t_0\le t_1<\cdots<t_N$,
\begin{equation*}
  \mathbb{V}_{\lambda(x^*_0,\cdot)}
  (X^*(t_i)\in dx_i^*,i=1,\ldots,N)
  = \prod_{i=1}^N Q^*_{t_i-t_{i-1}}(x_{i-1}^*,dx_i^*)
\end{equation*}
by applying (\ref{v.link}) recursively.
\end{proof}

We call the probability measure $\mathbb{V}_{\lambda(x^*,\cdot)}$
of Theorem~\ref{link.thm}
a \emph{$\Lambda$-linked coupling}.
It satisfies
\begin{equation*}
  \mathbb{V}_{\lambda(x^*,\cdot)}
  ((X^*(t),X(t))\in dy^*\times{dy})
  = Q^*_t(x^*,dy^*)\lambda(y^*,y)\,dy ,
\end{equation*}
which allows us to derive a regular conditional probability
\begin{equation}\label{vee.cond}
  \mathbb{V}_{\lambda(x^*,\cdot)}
  (X(t)\in dy| X^*(t)=y^*)=\lambda(y^*,y)dy
\end{equation}
for $t> 0$.

\subsection{Liggett dual}

Let $Q_t$ be a sub-Markov semigroup on $D^*$,
and let $\Gamma(x^*,x)$ be a bounded nonnegative measurable
function on $D^*\times D$.
Then $Q_t$ is said to be a \emph{Liggett dual} of $P_t$ with
respect to $\Gamma$ if
\begin{equation*}
\int P_t(x,dy)\Gamma(x^*,y)
= \int Q_t(x^*,dy^*)\Gamma(y^*,x)
\end{equation*}
holds for $(x^*,x)\in D^*\!\times\!D$.

\begin{proposition}\label{twin.prop}
Suppose that $Q_t$ is a Liggett dual of $P_t$
with respect to $\Gamma$,
and that
\begin{equation}\label{hgv}
  h(x^*) = \int\Gamma(x^*,x)\nu(x)\,dx
\end{equation}
is finite and strictly positive on $D^*$.
Then the Markov semigroup
\begin{equation*}
  Q^*_tf(x^*)
  = \frac{1}{h(x^*)}Q_t[hf](x^*)
\end{equation*}
is an intertwining dual of $P_t$
with respect to the link
\begin{equation*}
\lambda(x^*,x)
= \frac{\Gamma(x^*,x)}{h(x^*)}\nu(x)
\end{equation*}
\end{proposition}

By applying
the Liggett duality and the $\nu$-symmetry of $P_t$,
we can observe that
\begin{align*}
Q_t h
& =
\int Q_t(x^*,dy^*)\int\Gamma(y^*,x)\nu(x)\,dx
\\ & =
\int\Gamma(x^*,y)\,dy\int \nu(x)p(t,x,y)\,dx
= h(x^*) ;
\end{align*}
thus, $h$ is harmonic for $Q_t$.
Given the harmonic function $h$,
the semigroup $Q^*_t$ of Proposition~\ref{twin.prop}
is known as the
\emph{Doob h-transform},
and it is clearly conservative.

\begin{proof}
We obtain
\begin{align*}
\Lambda[P_t f](x^*)
& = \int \frac{\Gamma(x^*,x)}{h(x^*)}\nu(x)\,dx
\int p(t,x,y)f(y)\,dy
\\
& = \frac{1}{h(x^*)}\int \nu(y)f(y)\,dy
\int \Gamma(x^*,x) p(t,y,x)\,dx
\\
& = \frac{1}{h(x^*)}\int Q_t(x^*,dy^*)h(y^*)
\int \frac{\Gamma(y^*,y)}{h(y^*)}\nu(y)f(y)\,dy
\\
& = \frac{1}{h(x^*)}Q_t[h(\Lambda f)](x^*)
= Q^*_t[\Lambda f](x^*) .
\end{align*}
Hence, $P_t$ and $Q^*_t$ are $\Lambda$-linked.
\end{proof}

The Liggett dual $Q_t$ of Proposition~\ref{twin.prop} may
not be conservative, but it can be extended to
a Markov semigroup
over $\bar{D}^* = D^*\cup\partial$.
We can generate a Markov process
$X^*$ on $\bar{D}^*$ by $Q_t$ with exit boundary $\partial$.
If $Q_t$ is conservative, no Markov process started in $D^*$ reaches
the coffin state.
By setting the terminal time
$\zeta = \inf\{t\ge 0: X^*(t)\in\partial\}$
accompanied with $X^*$,
we can view it as a Markov process $X^*(t)$ over $D^*$
defined for the duration $[0,\zeta)$.
For a duality function $\Gamma$
it is understood customarily that
$\Gamma(X^*(t),x) = 0$ if $t\ge\zeta$,
or equivalently that
$\Gamma$ is extended over $\bar{D}^*\times D$
by setting $\Gamma(x^*,x) = 0$ for all $x^*\in\partial$.

In the next proposition we consider
the link $\lambda$ of Proposition~\ref{twin.prop},
and
the probability measure $\mathbb{P}_x$ on $C(\mathbb{R}_+,D)$
corresponding to the semigroup $P_t$.
We also assume that
$E = \{(x^*,x)\in D^*\times D: \Gamma(x^*,x)>0\}$
is Polish.
Let $\Psi^*_t$ be a dynamical system of Definition~\ref{mds}
from $\bar{D}^*\times C([0,t],D)$ to $\bar{D}^*$.

\begin{proposition}\label{s.twin.prop}
Assume that
(a)
for any $x^*\in \bar{D}^*$,
$X\in C(\mathbb{R}_+,D)$, and $t> 0$,
\begin{equation}\label{s.ligg.g}
  \Gamma(x^*,X(0))
  = \Gamma(\Psi^*_t(x^*,X),X(t))
\end{equation}
holds,
(b) for each $t > 0$
\begin{equation*}
  Q_tf(x^*) = \mathbf{E}_{\mathbb{P}_x}[f(\Psi^*_t(x^*,X(t-\cdot))] ,
  \quad x^*\in \bar{D}^*
\end{equation*}
defines a Markov semigroup on $\bar{D}^*$
uniquely regardless of the initial state $X(0)=x$,
and (c)
$(\Psi^*_t(x^*,X),X(t))$ is Markovian
whenever $X(t)$ is Markovian.
Then the semigroup
\begin{equation*}
V_t g(x^*,x)
= \mathbf{E}_{\mathbb{P}_x}[g(\Psi^*_t(x^*,X),X(t))],
\quad (x^*,x)\in E,
\end{equation*}
is $\Lambda$-linked between $P_t$ and $Q^*_t$ of
Proposition~\ref{twin.prop}.
\end{proposition}

\begin{proof}
Obviously we have $V_t f = P_t f$
if $f(x^*,x) = f(x)$ on $E$.
To show (\ref{v.link}),
we apply (a) and obtain
for each $(x^*,x)\in E$
\begin{align*}
  & \int\lambda(x^*,x)
  \mathbf{E}_{\mathbb{P}_x}[g(\Psi^*_t(x^*,X),X(t))]\,dx
  \\
  & \hspace{0.2in}
  = \frac{1}{h(x^*)}\int \nu(x)
  \mathbf{E}_{\mathbb{P}_x}[
    \Gamma(\Psi^*_t(x^*,X),X(t))g(\Psi^*_t(x^*,X),X(t))
  ]\,dx
\end{align*}
By Lemma~\ref{rev.lem}
we can further reduce the above integration to
\begin{align*}
  & \frac{1}{h(x^*)}\int \nu(y)
  \mathbf{E}_{\mathbb{P}_y}[
  \Gamma(\Psi^*_t(x^*,X(t-\cdot)),X(0))
  g(\Psi^*_t(x^*,X(t-\cdot)),X(0)]\,dy
  \\ & \hspace{0.2in}
  = \frac{1}{h(x^*)}\int \nu(y)\,dy
  \int\Gamma(y^*,y)g(y^*,y)
  Q_t(x^*,dy^*)
  \\ & \hspace{0.2in}
  = \int\frac{h(y^*)}{h(x^*)}Q_t(x^*,dy^*)
  \int\frac{\nu(y)}{h(y^*)}\Gamma(y^*,y)g(y^*,y)\,dy
  \\ & \hspace{0.2in}
  = \int Q^*_t(x^*,dy^*)\int\lambda(y^*,y)g(y^*,y)\,dy,
\end{align*}
which completes the proof.
\end{proof}

We call $\Psi^*_t$ of Proposition~\ref{s.twin.prop}
\emph{$\Lambda$-linked}.
By setting $f(y^*)=\Gamma(y^*,x)$ in
Proposition~\ref{s.twin.prop}(b)
we can observe that
\begin{equation*}
  Q_t\Gamma(\cdot,x)(x^*)
  = \mathbf{E}_{\mathbb{P}_x}[\Gamma(\Psi^*_t(x^*,X(t-\cdot),x)]
  = \mathbf{E}_{\mathbb{P}_x}[\Gamma(x^*,X(t))] .
\end{equation*}
Thus,
the existence of $\Lambda$-linked dynamical system
$\Psi^*_t$ of Proposition~\ref{s.twin.prop}
implies that $Q_t$ is a Liggett dual of $P_t$.

\section{Stochastic processes by approximation}
\label{euler.app}

Let $\Phi_t$ be a Markov dynamical system 
from $\mathbb{R}^n\times C([0,t],\mathbb{R}^n)$ to $\mathbb{R}^n$
driven by the Wiener measure $\mathbb{W}$.
We call $\Phi_t$ a \emph{strong solution}
if $X(t) = \Phi_t(x,W)$
is a unique solution to~(\ref{c.sde})
with initial condition $X(0) = x$
(cf. Theorem~IV-1.1 of \cite{iw}).
Such a strong solution $\Phi_t$ exists
if $\beta$ is locally Lipschitz
continuous (cf. Theorem~IV-3.1 of \cite{iw}).
In this paper we assume that
the drift coefficient $\beta$ is smooth with bounded first
derivatives;
thus, it has a Lipschitz constant $K_\beta$.

Here we view (\ref{c.sde}) as a time-reversed process $\hat{X}(s)$
backward in time $s$.
The strong solution $\Phi_s$
forms a diffeomorphic map $\Phi_s(\cdot,\hat{\omega})$
from $\mathbb{R}^n$ to itself for each $s\ge 0$ and
$\hat{\omega}\in C(\mathbb{R}_+,\mathbb{R}^n)$ (cf. Theorem~V-13.8
of~\cite{rw2}).
By $\Phi_s^{-1}(\cdot,\hat{\omega})$ we denote the inverse
map of $\Phi_s(\cdot,\hat{\omega})$ for each
$\hat{\omega}\in C(\mathbb{R}_+,\mathbb{R}^n)$.

\subsection{Euler approximation of backward process}
\label{euler.scheme}

Here we fix $T > 0$,
and develop an approximation $\hat{X}_N(s)$
for an $n$-dimensional drifting Brownian motion~(\ref{c.sde}).
Set a uniform increment
$0 = s_0 < \cdots < s_N = T$, and define a map $\phi_u$ by
\begin{equation}\label{phi.euler}
  \phi_u(z,\hat{\omega})
  = z - \beta(z)u + \hat{\omega}(u) - \hat{\omega}(0).
\end{equation}
Starting at $\hat{X}_N(0) = x_N$,
we can recursively construct
\begin{equation}\label{x.euler}
\hat{X}_N(s)
= \phi_{s-s_{k-1}}(\hat{X}_N(s_{k-1}),\hat{\omega}(\cdot+s_{k-1}))
\end{equation}
for $s_{k-1} < s \le s_k$, $k=1,\ldots,N$.
By using $a\wedge b = \min\{a,b\}$
and $[c]_+ = \max\{c,0\}$,
we can formulate (\ref{x.euler}) as
\begin{align}\label{x.sum.euler}
  & \hat{X}_N(s) = x_N
  - \sum_{k=1}^{N} \beta(\hat{X}_N(s_{k-1}))
  [(s-s_{k-1})\wedge(s_k - s_{k-1})]_+
  \\ \nonumber & \hspace{40ex}  
  + \hat{\omega}(s) - \hat{\omega}(0).
\end{align}
Assuming that $x_N$ is convergent,
the approximation of~(\ref{x.euler})
or~(\ref{x.sum.euler}) is known to converge,
and called an explicit Euler method for numerical solutions of
SDE's; see, e.g., Kloeden and Platen~\cite{kloeden}.

For $f\in C([0,T],\mathbb{R}^n)$ we define the modulus of continuity
(cf. Chapter~2 of~\cite{billingsley}) by
\begin{equation*}
  \Delta_{\delta}f
  = \sup\{\|f(s+u)-f(s)\|: 0\le s< s+u\le T, u\le\delta\},
  \quad 0<\delta\le T.
\end{equation*}
We also set
\begin{equation*}
  \|f\|_s
  = \sup\{\|f(u)\|: 0\le u\le s\},
  \quad 0\le s\le T,
\end{equation*}
and write $|f|_s$ instead of $\|f\|_s$
when $f$ is a scalar function.
The lemma below shows
uniform boundedness and equicontinuity
of the approximation $\hat{X}_N$.

\begin{lemma}\label{x.unif}
For any $\delta>0$ we have
\begin{align*}
  \|\hat{X}_N\|_T
  & \le e^{K_\beta T}(\|x_N\|+\|\beta(0)\|T+2\|\hat{\omega}\|_T);
  \\
  \Delta_\delta\hat{X}_N
  & \le (K_\beta \|\hat{X}_N\|_T
  + \|\beta(0)\|)\delta + \Delta_\delta\hat{\omega} .
\end{align*}
\end{lemma}

\begin{proof}
Since $\|\beta(\hat{X}_N(s))\|$
is bounded by $K_\beta \|\hat{X}_N\|_T + \|\beta(0)\|$,
by~(\ref{x.sum.euler})
we obtain an upper bound for $\Delta_\delta\hat{X}_N$.
Observe for $k=0,\ldots,N$ that
\begin{equation*}
  \|\hat{X}_N(s_k)\|
  \le \|x_N\| + \|\beta(0)\|T + 2\|\hat{\omega}\|_T
  + K_\beta\sum_{i=0}^{k-1}\|\hat{X}_N(s_i)\|(s_{i+1}-s_i) .
\end{equation*}
Then the upper bound for $\|\hat{X}_N\|_T$
is an immediate consequence of
the following version of discrete Gronwall's inequality:
If $x_k\le \alpha + \sum_{i=0}^{k-1}\gamma_i x_i$
for $k=0,\ldots,N$ with nonnegative $\alpha$ and $\gamma_i$'s
then we have
$x_k\le \alpha\exp\left(\sum_{i=0}^{k-1}\gamma_i\right)$
for $k=0,\ldots,N$.
\end{proof}

Assuming that $x_N$ converges to $x$,
a subsequence of $\{\hat{X}_N\}$
converges uniformly by Ascoli-Arzel\`{a} theorem.
The limiting process $\hat{X}$ implies the existence of
solution to
\begin{equation}\label{c.ie}
  \hat{X}(s) = x
  - \int_0^s \beta(\hat{X}(u))\,du + \hat{\omega}(s)
  - \hat{\omega}(0) .
\end{equation}
Since the solution must be unique,
the whole sequence $\{\hat{X}_N\}$ must converge
uniformly to $\hat{X}$.
Clearly Lemma~\ref{x.unif} holds for $\hat{X}$.

\subsection{Implicit Euler scheme of forward process}

Since the strong solution $\Phi_s(\cdot,\hat{\omega})$
to (\ref{c.sde}) is a diffeomorphism,
we use the backward sample path $\omega(t-u)$, $0\le u\le t$,
and obtain a strong solution
$X(t) = \Phi_t^{-1}(x,\omega(t-\cdot))$ to
\begin{equation}\label{h.sde}
  X(t) = x
  + \int_0^t \beta(X(v))\,dv + \omega(t) - \omega(0) .
\end{equation}

We set the forward time increment
$t_j = T - s_{N-j}$, $j=0,\ldots,N$.
Here we assume a sufficiently small increment $\delta > 0$
(i.e., a sufficiently large $N$)
so that $\phi_u(\cdot,\hat{\omega})$ is diffeomorphic
for every $0\le u< \delta$.
Starting from $X_N(0) = x_N$,
we can formulate an approximation $X_N(t)$ to~(\ref{h.sde})
recursively by
\begin{equation}\label{x.fwd}
X_N(t) =
\phi_{t_j-t}\left(
  \phi_{t_j-t_{j-1}}^{-1}(X_N(t_{j-1}),\omega(t_j-\cdot)),
  \omega(t_j-\cdot)
\right)
\end{equation}
for $t_{j-1}<t \le t_j$, $j=1,\ldots,N$,
as if the time-reversed process
$\hat{X}_N(s) = X_N(T-s)$, $0\le s\le T$,
were generated by~(\ref{x.euler})
with terminal condition $\hat{X}_N(T) = x_N$.
It is equivalently formulated by
\begin{equation}\label{x.sum.fwd}
  X_N(t) = x_N
  + \sum_{j=1}^{N} \beta(X_N(t_j))
  [(t-t_{j-1})\wedge(t_j - t_{j-1})]_+ + \omega(t) - \omega(0),
\end{equation}
and called an \emph{implicit Euler scheme}.

Provided that $x_N$ converges to $x$,
the uniform convergence of $X_N$ is an immediate consequence
to the following lemma.

\begin{lemma}\label{x.unif.fwd}
Let $\delta>0$ be arbitrarily fixed.
For a sufficiently large $N$ we have
\begin{align*}
  \max_{0\le j\le N}\|X_N(t_j)\|
  & \le 2 e^{2K_\beta T}
  (\|x_N\|+\|\beta(0)\|T+2\|\omega\|_T);
  \\
  \|X_N\|_T
  & \le \|x_N\|+2\|\omega\|_T
  +\left(
  \|\beta(0)\|+K_\beta\max_{0\le j\le N}\|X_N(t_j)\|
  \right)T;
  \\
  \Delta_\delta X_N
  & \le \left(
  \|\beta(0)\|+K_\beta\max_{0\le j\le N}\|X_N(t_j)\|
  \right)\delta
  + \Delta_\delta \omega .
\end{align*}
\end{lemma}

\begin{proof}
By (\ref{x.sum.fwd}) we obtain for $j=1,\ldots,N$,
\begin{equation*}
  \|X_N(t_j)\|
  \le \|x_N\| + \|\beta(0)\|T + 2\|\omega\|_T
  + K_\beta \sum_{i=1}^j \|X_N(t_i)\|(t_i-t_{i-1})
\end{equation*}
For a sufficiently large $N$
we can find that $[1-K_\beta(t_j-t_{j-1})]\ge 1/2$
for every $j$, and that
\begin{equation*}  \|X_N(t_j)\|
  \le 2\left[
    \|x_N\| + \|\beta(0)\|T + 2\|\omega\|_T
    + K_\beta \sum_{i=1}^{j-1} \|X_N(t_i)\|(t_i-t_{i-1})
  \right]
\end{equation*}
The rest of the proof is completed similarly to that of
Lemma~\ref{x.unif}.
\end{proof}

\subsection{Stochastic processes of inverse image}

By $\mathbb{F}_0$ we denote the space of nonempty closed subsets in
$\mathbb{R}^n$.
Equipped with Fell topology $\mathbb{F}_0$ is a Polish space,
which is also characterized by Painlev\'{e}-Kuratowski
convergence (cf. Appendix~B of \cite{molchanov}).
Let $\{x^*_N\}$ be a sequence in $\mathbb{F}_0$.
The lower limit, denoted by $\liminf x^*_N$, consists of all the
points $x$ such that $x_N\to x$ with $x_N\in x^*_N$.
The upper limit, denoted by $\limsup x^*_N$,
consists of all the limiting points $x$ of some subsequence
$\{x_{N_k}\}$ with $x_{N_k}\in x^*_{N_k}$.
The sequence $\{x^*_N\}$ converges to $x^*$ in the
Painlev\'{e}-Kuratowski sense,
denoted by $x^* = \mbox{{\rm\tiny PK}-}\!\lim x^*_N$,
if $x^* = \liminf x^*_N = \limsup x^*_N$.

An $\mathbb{F}_0$-valued function $F(t)$ is
lower semicontinuous at $t=t_0$ if
$\liminf F(t_N)\supseteq F(t_0)$
for every sequence $\{t_N\}$ converging to $t_0$,
and it is upper semicontinuous at $t=t_0$ if
$\limsup F(t_N)\subseteq F(t_0)$
for each of such sequences
(cf. Appendix~D of \cite{molchanov}).
Then $F(t)$ is continuous
if it is lower and upper semicontinuous.
In the next theorem
we consider a dynamical system $\psi_t$
from $\mathbb{R}^n\times C([0,t],\mathbb{R}^n)$
to $\mathbb{R}^n$,
and introduce a sufficient condition for upper semicontinuity
when the $\mathbb{F}_0$-valued process of inverse image is
constructed.

\begin{definition}\label{consist}
Let $\{s_N\}$ be a decreasing or an increasing sequence in
$\mathbb{R}_+$.
Then we call a dynamical system $\psi_t$ \emph{consistent}
if
$\psi_{s-s_N}(x_N,\hat{\omega}(\cdot+s_N))$
converges to
$\psi_{s-s_0}(x_0,\hat{\omega}(\cdot+s_0))$
for each $s > s_0$
whenever $(s_N,x_N)$ converges to $(s_0,x_0)$.
\end{definition}

\begin{theorem}\label{h.cont}
Let $x^*\in\mathbb{F}_0$ and
$\omega\in C(\mathbb{R}_+,\mathbb{R}^n)$
be fixed,
and let
$\psi_t$ be a consistent dynamical system of
Definition~\ref{consist}.
Assuming that $Y^*(t) = \psi_t^{-1}(x^*,\omega(t-\cdot))$
is nonempty for all $t\ge 0$,
$Y^*(t)$ is an upper semicontinuous 
$\mathbb{F}_0$-valued process.
\end{theorem}

\begin{proof}
Definition~\ref{consist} implies that
$\psi_t(x,\omega(t-\cdot))$ is continuous in $x$,
and therefore, that
$Y^*(t)$ takes values on $\mathbb{F}_0$.
In order to show upper semicontinuity,
we set a sequence $\{(t_N,z_N)\}$
converging $(t_0,z_0)$
while $\psi_{t_N}(z_N,\omega(t_N-\cdot))\in x^*$.
Without loss of generality
we assume that $\{t_N\}$ is increasing or decreasing.
By Definition~\ref{consist}
we can observe that
\begin{math}
  \lim_{N\to\infty}\psi_{t_N}(z_N,\omega(t_N-\cdot))
  = \psi_{t_0}(z_0,\omega(t_0-\cdot)),
\end{math}
and therefore, that $z_0\in Y^*(t_0)$.
\end{proof}

In the case of strong solution $\Phi_s$
in Section~\ref{euler.scheme}
we obtain a continuous $\mathbb{F}_0$-valued process
of inverse image.

\begin{lemma}\label{ss.consist}
The strong solution $\Phi_s$ to (\ref{c.sde})
is consistent.
\end{lemma}

\begin{proof}
Suppose that $(s_N,x_N)$ converges to $(s_0,x_0)$
as in Definition~\ref{consist},
and that $T > s_0$ is arbitrarily fixed.
Then the strong solution
$\hat{X}_N(s) = \Phi_{s-s_N}(x_N,\hat{\omega}(\cdot+s_N))$,
$s_N\le s\le T$, satisfies
\begin{equation*}
  \hat{X}_N(s) = x_N
  - \int_{s_N}^s \beta(\hat{X}_N(u))\,du + \hat{\omega}(s)
  - \hat{\omega}(s_N) ,
\end{equation*}
and a sequence $\{\hat{X}_N(s)\}$
of processes on $[s_0\vee s_1, T]$
is uniformly bounded and equicontinuous for each $N$.
By extending $\hat{X}_N(s)$ to a continuous process on $(s_0,T]$
as necessary, we can find a subsequence
$\{\hat{X}_{N_i}(s)\}$ which converges to $\hat{X}(s)$
uniformly on $(s_0,T]$.
Since $\hat{X}(s)$ must satisfy
\begin{equation*}
  \hat{X}(s) = x_0
  - \int_{s_0}^s \beta(\hat{X}(u))\,du + \hat{\omega}(s)
  - \hat{\omega}(s_0) ,
\end{equation*}
we obtain
$\hat{X}(s) = \Phi_{s-s_0}(x_0,\hat{\omega}(\cdot+s_0))$.
The uniqueness of integral equation implies that
the whole sequence $\{\hat{X}_N(s)\}$ must converge
to $\hat{X}(s)$ uniformly on $(s_0,T]$.
\end{proof}

\begin{proposition}\label{ss.cont}
Let $x^*\in\mathbb{F}_0$ and
$\omega\in C(\mathbb{R}_+,\mathbb{R}^n)$
be fixed.
Then $Y^*(t) = \Phi_t^{-1}(x^*,\omega(t-\cdot))$
is a continuous 
$\mathbb{F}_0$-valued process.
\end{proposition}

\begin{proof}
By Theorem~\ref{h.cont} and Lemma~\ref{ss.consist} $Y^*(t)$ is upper
semicontinuous.
The proof of lower semicontinuity is naturally related to a
selection of continuous process $X(t)\in Y^*(t)$
(cf. Chapter~9 of~\cite{aubin90}).
Suppose that $x_0\in Y^*(t_0)$,
and that a sequence $\{t_N\}$ converges to $t_0$.
Then we can construct a strong solution
\begin{equation*}
  X(t) = \begin{cases}
    \Phi_{t_0-t}(x_0,\omega(t_0-\cdot))
    & \mbox{ if $0\le t< t_0$; } \\
    \Phi_{t-t_0}^{-1}(x_0,\omega(t-\cdot))
    & \mbox{ if $t_0\le t< \infty$; }
  \end{cases}
\end{equation*}
to (\ref{h.sde}) satisfying $X(t_0) = x_0$.
Thus, $x_N = X(t_N) \in Y^*(t_N)$ converges to $x_0$,
and therefore, $Y^*(t)$ is lower semicontinuous.
\end{proof}

A process $Y^*(t)$ of Proposition~\ref{ss.cont}
is viewed as a Markov process constructed by
Markov dynamical system
$\Phi_t^{-1}(x^*,\omega(t-\cdot))$.
An approximation $Y^*_N(t)$ to $Y^*(t)$
starts with $Y^*_N(0) = Y^*(0)$
and is recursively updated by
\begin{equation}\label{h.euler}
Y^*_N(t) =
\phi_{t_j-t}\left(
  \phi_{t_j-t_{j-1}}^{-1}(Y^*_N(t_{j-1}),\omega(t_j-\cdot)),
  \omega(t_j-\cdot)
\right)
\end{equation}
for $t_{j-1}<t \le t_j$, $j=1,\ldots,N$.
Then
it can be viewed as the collection of
approximations of (\ref{x.fwd})
starting from $X_N(0)\in Y^*(0)$.
It can be verified that
$Y^*(t) = \mbox{{\rm\tiny PK}-}\!\lim_{N\to\infty} Y^*_N(t)$.

\section{Stochastic processes of hypographical surface}
\label{hypo.sec}

By $X_i$ or $X_{i,N}$
we denote the $i$-th coordinate of $X$ or $X_N$,
and by $X_{(-i)}$ or $X_{(-i),N}$
we denote the $(n-1)$-dimensional vector of $X$ or $X_N$,
by deleting the $i$-th coordinate.
Assuming $n \ge 2$,
a closed subset $x^*$ of $\mathbb{R}^n$
is said to be \emph{hypographic} at the direction of
$i$-th coordinate if there exists a unique upper semi-continuous function
$h$ from $\mathbb{R}^{n-1}$ to $\mathbb{R}$ satisfying
\begin{math}
  x^* = \{(x_i;x_{(-i)}):
  x_i\le h(x_{(-i)}),\,x_{(-i)}\in\mathbb{R}^{n-1}\} .
\end{math}
We call a hypographic closed subset $x^*$
``Lipschitz-continuous''
if the corresponding function $h$
of hypograph is Lipschitz-continuous.
The boundary $\partial x^*$ of hypographic closed set
uniquely determines the hypographical surface
\begin{equation}\label{graph}
  \partial x^*
  = \{x\in\mathbb{R}^n: x_i = h(x_{(-i)})\} .
\end{equation}
Thus, 
we denote by $\partial x^*(\cdot)$
the corresponding function $h(\cdot)$ in (\ref{graph}),
though it is a slight abuse of notation.

\begin{definition}\label{y.star.def}
We denote by $\mathbb{F}_1$ a subclass of
Lipschitz-continuous hypographic closed sets
at the direction of first coordinate,
and assume that a process
$Y^*(t) = \Phi^{-1}_t(x^*,\omega(t-\cdot))$ of inverse image
takes values on $\mathbb{F}_1$.
Then the corresponding function of hypographical surface
at each $t$, denote by $\partial Y^*(t,\cdot)$,
is called a \emph{stochastic process of hypographical surface}
if for each $T\ge 0$
there exists a Lipschitz constant universally for
the collection $\{\partial Y^*(t,\cdot)\}_{0\le t\le T}$
regardless of $\omega\in C([0,T],\mathbb{R}^n)$.
\end{definition}

\subsection{Examples of hypographical surface}
\label{m1.sec}

A trivial example of hypographical surfaces is given by
\begin{equation*}
  \mathbb{F}_1 = \left\{
  \mbox{ $x^*$: hypographic with constant $\partial x^*(\cdot)$ }
  \right\}
\end{equation*}
if $n=1$ or stochastic processes
$X_1$ and $X_{(-1)}$ are independent.

\begin{example}\label{r1.ex}
Suppose that $n=1$ or 
an invariant function $\nu$ is formulated by
$\nu(x)=\nu_1(x_1)\nu_2(x_{(-1)})$.
Then the drift coefficient $\beta_1(x)$ is a function of $x_1$,
and the hypographical surface $\partial Y^*(t,\cdot)\equiv Y_1(t)$
is determined by
$dY_1(t) = \beta_1(Y_1(t))dt + dW_1(t)$.
\end{example}

In general
neither the initial condition of $Y^*(0)\in \mathbb{F}_1$
nor the Lipschitz-continuity of $\beta$
ensures that a closed set-valued process $Y^*(t)$
remains on a subclass $\mathbb{F}_1$ of hypographic closed sets.

\begin{example}\label{r2.ex}
Let
$\nu(x) = e^{-2x_1 x_2}$ be an invariant function on
$\mathbb{R}^2$.
Then we can find the drift coefficient
$\beta(x) = [x_2, x_1]^T$,
and set
\begin{equation*}
  \mathbb{F}_1 = \left\{
  \mbox{ $x^*$: hypographic with
    $\partial x^*(x_2) = (\tanh\theta)x_2 + \eta$
    for $\theta,\eta\in\mathbb{R}$
  }
  \right\} .
\end{equation*}
A direction $U(t) = [\sinh(\theta+t),\cosh(\theta+t)]$
of the line $\partial Y^*(t)$ satisfies
$\frac{dU}{dt} = \beta(U)$ if
$U(0) = [\sinh\theta,\cosh\theta]$.
Thus, given an initial direction $U(0)$ and
a point $Y(0)$ on the line $\partial Y^*(0)$,
the hypographical surface
\begin{math}
  \partial Y^*(t,x_2)
  = (U_1(t)/U_2(t))(x_2-Y_2(t))+Y_1(t)
\end{math}
is determined by
$dU(t) = \beta(U(t))dt$
and $dY(t) = \beta(Y(t))dt + dW(t)$.
\end{example}

In the next example
we view $D = \mathbb{R}^n$ as a parameter space,
and consider a linear regression
$\langle a, x\rangle = \sum_{i=1}^n a_i x_i$
with vector $a = [a_1,\ldots,a_n]^T$ of explanatory variables.
By $S(\theta) = 1/(1 + e^{-\theta})$
we denote the logistic sigmoid function.
Then we can generate a binary output $b = 0$ or $1$
according to the probability
$S((2b-1)\langle a, x\rangle)$,
and call it
a \emph{Bernoulli-logistic regression model}.
In the neural network terminology
this is a unit perceptron
with input $a$ and weight vector $x$.
Provided a training data set
$\{(a^{(1)},b^{(1)}),$ $\ldots,$ $(a^{(N)},b^{(N)})\}$
consisting of $N$ input-output pairs,
we can construct the likelihood function
\begin{align*}
  \nu(x)
  & = \prod_{j=1}^N S((2b^{(j)}-1)\langle a^{(j)}, x\rangle)
  \\
  & = \exp\left(
  \sum_{j=1}^N b^{(j)}\langle a^{(j)}, x\rangle
  \right)
  \times\prod_{j=1}^N S(-\langle a^{(j)}, x\rangle)
\end{align*}
by applying
\begin{math}
  S(\langle a, x\rangle)^b
  S(-\langle a, x\rangle)^{1 - b}
  = \exp(b \langle a, x\rangle)
  S(-\langle a, x\rangle) .
\end{math}

\begin{example}\label{b.model}
We can consider the above likelihood function
as an invariant function,
and obtain the drift coefficient
\begin{equation*}
\beta(x)
= \frac{1}{2}\sum_{j=1}^N
a^{(j)}\left[
S(\langle a^{(j)}, x\rangle) - b^{(j)}
\right] .
\end{equation*}
Suppose that the input vectors
$a^{(1)},\ldots,a^{(N)}$ span an $(n-1)$-dimensional subspace $H$,
and that a unit normal vector $d$ to the subspace $H$
has a positive component $d_1 = \cos\theta_0$
to the first coordinate.
Furthermore, it is assumed that
\begin{equation}\label{nu.finite}
  \sup_{x\in H,\|x\|=1}\min_{j=1,\ldots,N}
  (2b^{(j)}-1)\langle a^{(j)}, x\rangle
  < 0
\end{equation}
so that $\int_H\nu(x)dx < \infty$;
thus, in Bayesian viewpoint
the function $\nu(x)$ on the subspace $H$
is proportional to the posterior density function $\nu_H(x)$
given the flat prior.
Here we can choose $\theta_0<\theta\le\pi/2$,
and introduce a subclass $\mathbb{F}_1$ of hypographic closed subsets
$x^*$ satisfying
\begin{math}
  |\langle d,x-y\rangle|
  \le \|x-y\||\cos\theta|
\end{math}
for any $x,y\in\partial x^*$.
A Lipschitz constant for $\partial x^*$ in $\mathbb{F}_1$
is bounded by
\begin{math}
  (\|d_{(-1)}\|+\cos\theta)/(d_1-\cos\theta) .
\end{math}
Particularly we have
\begin{equation*}
  \mathbb{F}_1 = \{x^*:
  \mbox{ hypographic with $\partial x^* = H + c$
    with some $c\in\mathbb{R}$ } \}
\end{equation*}
if $\theta=\pi/2$.
\end{example}

In Example~\ref{b.model}
we can construct a process 
$Y^*(t) = \Phi^{-1}_t(x^*,\omega(t-\cdot))$
starting from $Y^*(0)=x^*\in\mathbb{F}_1$.
Consider two distinct paths
$X(t) = \Phi^{-1}_t(X(0),\omega(t-\cdot))$
and $Y(t) = \Phi^{-1}_t(Y(0),\omega(t-\cdot))$
on $\partial Y^*(t)$.
Then we can introduce the difference
$z(t) = X(t) - Y(t) = \alpha(t) d + h(t)$
in the coordinate system with the vector $d$
and the subspace $H$
by setting
$\alpha(t)=\langle d,z(t)\rangle$
and $h(t)=z(t)-\alpha(t) d$.
Observe that
$z(t)$ is a solution to the differential equation
\begin{equation*}
  \frac{dz}{dt}
  = \frac{1}{2}\sum_{j=1}^N a^{(j)}\left[
    S(\langle a^{(j)}, z+X(t)\rangle)
    - S(\langle a^{(j)}, X(t)\rangle)
    \right] ,
\end{equation*}
and therefore, that
$\alpha(t) \equiv \alpha(0)$
and $\|h(t)\|$ is increasing.
Thus, we obtain
$|\alpha(t)|\le\|z(t)\| |\cos\theta|$,
which implies that $Y^*(t)\in\mathbb{F}_1$.

\subsection{A coupled approximation of backward process}
\label{y.cons.sec}

By $\phi_{i,s}$ or $\hat{\omega}_i$
we denote the $i$-th coordinate of $\phi_s$ or $\hat{\omega}$,
and by $\phi_{(-i),s}$ or $\hat{\omega}_{(-i)}$
the $(n-1)$-dimensional vector
of $\phi_s$ or $\hat{\omega}$ by deleting the $i$-th coordinate.
Since the maps $\phi_{i,s}(\cdot,\hat{\omega})$
and $\phi_{(-i),s}(\cdot,\hat{\omega})$
from~(\ref{phi.euler}) are determined respectively by $\hat{\omega}_i$
and $\hat{\omega}_{(-i)}$,
we can simply write
$\phi_{i,s}(\cdot,\hat{\omega}_i)$
and $\phi_{(-i),s}(\cdot,\hat{\omega}_{(-i)})$.
By $\hat{\omega}'$ we denote the $n$-dimensional sample path
\begin{equation}\label{hat.w'}
  \hat{\omega}'(s) = [-\hat{\omega}_1(s),\hat{\omega}_{(-1)}(s)]
\end{equation}
by changing the sign to the path $\hat{\omega}_1(s)$
of the first coordinate.

Let $T>0$ be fixed,
and let $\partial Y^*_N(t,\cdot)$, $0\le t\le T$,
be an approximated process of hypographical surface by~(\ref{h.euler})
with sample path $\omega'$ starting from
$\partial Y^*_N(0,\cdot) = \partial y^*(\cdot)$.
Then we set $\hat{\omega}(s) = \omega(T-s) - \omega(T)$,
and view $\hat{Y}^*_N(s) = Y^*_N(T-s)$, $0\le s\le T$,
as the time-reversed approximation by
\begin{equation}\label{y.hat.star}
  \hat{Y}^*_N(s)
  = \phi_{s-s_{k-1}}(\hat{Y}^*_N(s_{k-1}),
  \hat{\omega}'(\cdot+s_{k-1}))
\end{equation}
for $s_{k-1}\le s< s_k$, $k=1,\ldots,N$.
Furthermore,
we can build a backward process $\hat{Y}_N(s)$ of Algorithm~\ref{y.cons.alg}.

\begin{algorithm}\label{y.cons.alg}
Set $\hat{Y}_N(0) = x_N$ and $\hat{\sigma}_N(0) = 0$.
Provided $\hat{Y}_N(s_{k-1})$ and $\hat{\sigma}_N(s_{k-1})$,  
we can construct $\hat{Y}_N(s)$ and $\hat{\sigma}_N(s)$
recursively for $s_{k-1}< s\le s_k$, $k=1,\ldots,N$,
in the following steps:
(i)
Update $\hat{Y}_{(-1),N}(s)$ for $s_{k-1}< s\le s_k$ by
\begin{equation}\label{y-1.sigma}
  \hat{Y}_{(-1),N}(s)
  = \phi_{(-1),s-s_{k-1}}
  (\hat{Y}_N(s_{k-1}),
  \hat{\omega}_{(-1)}(\cdot+s_{k-1})) .
\end{equation}
(ii)
Update $\hat{\sigma}_N(s)$ for $s_{k-1}< s\le s_k$ by
\begin{equation}\label{s.out}
  \hat{\sigma}_N(s)
  = \hat{\sigma}_N(s_{k-1}) + 2\left(
    \hat{\omega}_1(s) - \hat{\omega}_1(s_{k-1})
  \right)
\end{equation}
if $\hat{Y}_N(s_{k-1})$ and $\hat{Y}_{(-1),N}(s_k)$ satisfy
\begin{align}\label{y.out}
  & \hat{Y}_{1,N}(s_{k-1})
  - \beta_1(\hat{Y}_N(s_{k-1}))(s_k-s_{k-1})
  + |\hat{\omega}_1(s_k) - \hat{\omega}_1(s_{k-1})|
  \\ \nonumber & \hspace{5ex}
  > \partial\hat{Y}^*_N(s_k,\hat{Y}_{(-1),N}(s_k));
\end{align}
otherwise, set
$\hat{\sigma}_N(s) \equiv \hat{\sigma}_N(s_{k-1})$.
(iii)
Complete the update of $\hat{Y}_N(s)$ by setting
\begin{equation}\label{y1.sigma}
  \hat{Y}_{1,N}(s)
  = \phi_{1,s-s_{k-1}}
  (\hat{Y}_N(s_{k-1}),
  (\hat{\omega}_1-\hat{\sigma}_N)(\cdot+s_{k-1}))
\end{equation}
for $s_{k-1}< s\le s_k$.
\end{algorithm}

Algorithm~\ref{y.cons.alg} constructs an approximation
$\hat{Y}_N(s)$ by (\ref{x.euler})
with backward sample path
\begin{equation}\label{w.sigma}
  \hat{\omega}^{\hat{\sigma}_N}(s)
  = [(\hat{\omega}_1-\hat{\sigma}_N)(s),
    \hat{\omega}_{(-1)}(s)],
  \quad 0\le s\le T,
\end{equation}
and couple it with $\hat{Y}^*_N(s)$ recursively
in such a way that
\begin{equation}\label{y.in}
  \hat{Y}_{1,N}(s_k) \le\partial\hat{Y}^*_N(s_k,\hat{Y}_{(-1),N}(s_k))
\end{equation}
for all $k=0,\ldots,N$ if
$x_N\in\hat{Y}^*_N(0)$.
Here we find $\hat{\sigma}_N(s_k)$ updated by (\ref{s.out})
only when
$\hat{\omega}_1(s_k) - \hat{\omega}_1(s_{k-1}) > 0$
holds; thus,
by~(\ref{y.out}) we have
\begin{equation}\label{y.at.k}
  \partial\hat{Y}^*_N(s_k,\hat{Y}_{(-1),N}(s_k))
  - 2(\hat{\omega}_1(s_k) - \hat{\omega}_1(s_{k-1}))
  < \hat{Y}_{1,N}(s_k)
\end{equation}
at the $k$-th update by (\ref{s.out}).

\begin{lemma}\label{weak.w}
If the sample path $\hat{\omega}$ of Algorithm~\ref{y.cons.alg}
is distributed as
$\mathbb{W}$ on $C([0,T],\mathbb{R}^n)$
then so is
$\hat{\omega}^{\hat{\sigma}_N}$ of~(\ref{w.sigma}).
\end{lemma}

\begin{proof}
We set for each $k=0,\ldots,N$
\begin{equation*}
  \hat{\omega}(s;k) = \begin{cases}
    \hat{\omega}^{\hat{\sigma}_N}(s) & \mbox{ if $0\le s\le s_k$; } \\
  [\hat{\omega}_1(s)-\hat{\sigma}_N(s_k),\hat{\omega}_{(-1)}(s)]
    & \mbox{ if $s_k< s\le T$. }
  \end{cases}
\end{equation*}
Observe that $\hat{\omega}(\cdot;0) = \hat{\omega}$
and $\hat{\omega}(\cdot;N) = \hat{\omega}^{\hat{\sigma}_N}$.
Then we can prove by induction that
$\hat{\omega}(\cdot;k)$ is distributed as $\mathbb{W}$
for each $k=1,\ldots,N$.

Suppose that $\hat{\omega}(\cdot;k-1)$ is distributed as
$\mathbb{W}$.
Then we obtain (i) $\hat{Y}_N(s_{k-1})$ from the initial state $x_N$
and the sample path $\hat{\omega}(s;k-1)$ on $[0,s_{k-1}]$,
and (ii) $Y^*_N(T-s_k)$ from the initial state $y^*$ and the sample path
$\omega(t) = \hat{\omega}(T-t;k-1)$ on $[0,T-s_k]$.
Having obtained $\hat{Y}_N(s_{k-1})$
and $\hat{Y}^*_N(s_k) = Y^*_N(T-s_k)$, 
we can determine whether (\ref{y.out}) holds or not
by the length
$|\hat{\omega}_1(s_k;k-1)-\hat{\omega}_1(s_{k-1};k-1)|$
and the vector
$\hat{\omega}_{(-1)}(s_k;k-1)-\hat{\omega}_{(-1)}(s_{k-1};k-1)$.
Hence, a sample path
\begin{equation*}
  \hat{\omega}_1^{\hat{\sigma}_N}(u+s_{k-1}) -
  \hat{\omega}_1^{\hat{\sigma}_N}(s_{k-1}),
  \quad 0 \le u\le s_k - s_{k-1},
\end{equation*}
is a Brownian motion independent of
$\hat{\omega}_1(s;k-1)$ on $[0,s_{k-1}]$,
$\hat{\omega}_1(s;k-1) - \hat{\omega}_1(s_k;k-1)$ on $[s_k,T]$,
and $\hat{\omega}_{(-1)}(s;k-1)$ on $[0,T]$.
Consequently, $\hat{\omega}(\cdot;k)$ is distributed as $\mathbb{W}$.
\end{proof}

We say that a sequence of stochastic processes is tight or weakly
converging if the sequence of their distributions is
tight or weakly converging
(cf. Chapter~2 of~\cite{billingsley}).
By Lemma~\ref{weak.w} we can find that $\hat{Y}_N$
of Algorithm~\ref{y.cons.alg}
is equal in distribution to
the approximation $\hat{X}_N$ by~(\ref{x.euler}).
If $\hat{X}_N$ converges weakly to $\hat{X}$,
so does $\hat{Y}_N$; thus, we obtain the following corollary.

\begin{proposition}\label{weak.y.hat}
Let $x_N$ converge to $x$ in Algorithm~\ref{y.cons.alg}.
Assuming the distribution of $\hat{\omega}$
as in Lemma~\ref{weak.w},
$\hat{Y}_N$ converges weakly to
the probability measure $\mathbb{P}_x$ of~(\ref{c.sde})
on $C([0,T],\mathbb{R}^n)$.
\end{proposition}

\subsection{Uniform boundedness and equicontinuity}

In the setting of Section~\ref{y.cons.sec}
we approximate $\partial Y^*_N(t,\cdot)$ by~(\ref{h.euler}),
$0\le t\le T$, with initial state
$\partial Y^*_N(0,\cdot) = \partial y^*$ and sample path $\omega'$.
By $K_{\partial Y^*}$
we denote a Lipschitz constant universally for
the limiting process $\partial Y^*(t,\cdot)$
of hypographical surface, $0\le t\le T$.
We assume that
$K_{\partial Y^*} \ge 1$, and
that $K_{\partial Y^*}$ is also a Lipschitz constant
for $\partial Y^*_N(t,\cdot)$, $0\le t\le T$.

\begin{lemma}\label{h.equicont}
Let $\{F_N(t)\}$ be a uniformly bounded and equicontinuous
sequence of $\mathbb{R}^{n-1}$-valued functions on $[0,T]$.
Then $\{\partial Y^*_N(t,F_N(t))\}$ is uniformly bounded and
equicontinuous on $[0,T]$.
\end{lemma}

\begin{proof}
For the boundedness
we start from $X_N(0) = [\partial y^*(0),0]$,
and construct $X_N(t)$ by~(\ref{x.fwd}) with sample path
$\omega'$.
Then we can observe that $X_N(t)\in\partial Y^*_N(t)$,
and that
\begin{align*}
  & |\partial Y^*_N(t,F_N(t))|
  \le |X_{1,N}(t)|
  + K_{\partial Y^*}\left(\|X_{N,(-1)}(t)\| + \|F_N(t)\|\right)
  \\ & \hspace{5ex}
  \le K_{\partial Y^*}\left(\sqrt{2}\|X_N(t)\| + \|F_N(t)\|\right).
\end{align*}
Since $X_N(t)$ and $F_N(t)$ are uniformly bounded on $[0,T]$,
so is $\partial Y^*_N(t,F_N(t))$.

Secondly for the equicontinuity
we start from
$X_N(0) = [\partial Y^*_N(t,F_N(t)),F_N(t)]$,
and construct $X_N(v)$
by~(\ref{x.fwd}) with sample path $\omega'(\cdot+t)$.
Let $\delta > 0$ be arbitrarily fixed.
By observing $X_N(\delta)\in\partial Y^*_N(t+\delta)$,
we obtain
\begin{align*}
  & |\partial Y^*_N(t,F_N(t))
  -\partial Y^*_N(t+\delta,F_N(t+\delta))|
  \\ & \hspace{5ex}
  \le |X_{1,N}(0) - X_{1,N}(\delta)|
  + |\partial Y^*_N(t+\delta,X_{(-1),N}(\delta))
  - \partial Y^*_N(t+\delta,F_N(t))|
  \\ & \hspace{5ex}
  + |\partial Y^*_N(t+\delta,F_N(t+\delta))
  - \partial Y^*_N(t+\delta,F_N(t))|
  \\ & \hspace{5ex}
  \le\sqrt{2}K_{\partial Y^*}\|X_N(0) - X_N(\delta)\|
  + K_{\partial Y^*}\|F_N(t) - F_N(t+\delta)\|
  \\ & \hspace{5ex}
  \le\sqrt{2}K_{\partial Y^*} \Delta_\delta X_N
  + K_{\partial Y^*} \Delta_\delta F_N .
\end{align*}
Therefore,
the equicontinuity of $\partial Y^*_N(t,F_N(t))$
is implied by that of $X_N(t)$ and $F_N(t)$.
\end{proof}

In the following lemma we consider
\begin{equation}\label{z.hat}
  \hat{Z}_N(s) = x_N
  - \sum_{k=1}^{N} \beta(\hat{Y}_N(s_{k-1}))
  [(s-s_{k-1})\wedge(s_k - s_{k-1})]_+
  + \hat{\omega}(s) - \hat{\omega}(0)
\end{equation}
for $0\le s\le T$.
Then we can observe that
$\hat{Y}_{1,N}(s) = \hat{Z}_{1,N}(s) - \hat{\sigma}_N(s)$
and
$\hat{Y}_{(-1),N}(s) = \hat{Z}_{(-1),N}(s)$.
Hence, the investigation of
uniformly boundedness and equicontinuity for
$\hat{Z}_N(s)$ allows us
to derive that of $\hat{\sigma}_N(s)$.

\begin{proposition}\label{y.tight}
Assuming that $x_N$ is bounded,
$\hat{\sigma}_N(s)$ is uniformly bounded and
equicontinuous on $[0,T]$.
\end{proposition}

\begin{proof}
If $x_N\not\in\hat{Y}^*_N(0)$
then the claim holds obviously
for $\hat{\sigma}_N = 2\hat{\omega}$;
thus, we assume in the proof that
$x_N\in\hat{Y}^*_N(0)$.
Let
$y_N(k)=\max_{0\le i\le k}\|\hat{Y}_N(s_i)\|$
and $z_N(k)=\max_{0\le i\le k}\|\hat{Z}_N(s_i)\|$
for $k=0,\ldots,N$.
Then we can find
\begin{equation*}
  z_N(k) \le x_N + 2\|\hat{\omega}\|_T
    +\|\beta(0)\|T + K_\beta\sum_{i=0}^{k-1}y_N(i)(s_{i+1}-s_i) .
\end{equation*}
Suppose that the last update by (\ref{s.out}) is completed
over the $i$-th interval $(s_{i-1},s_i]$ before $s_k$;
otherwise, set $i=0$.
Together with~(\ref{y.in}) and~(\ref{y.at.k}) we can show that
\begin{equation*}
  \hat{\sigma}_N(s_k) = \hat{\sigma}_N(s_i)
  \le |\hat{Z}_{1,N}(s_i)|
  + |\partial\hat{Y}^*_N(s_i,\hat{Z}_{(-1),N}(s_i))|
  + 2\Delta_\delta\hat{\omega}_1
\end{equation*}
By using the construction of $X_N(t)$ for the boundedness proof
of Lemma~\ref{h.equicont}
we obtain
\begin{equation}\label{sigma.hat}
  \hat{\sigma}_N(s_k)
  \le \sqrt{2}K_{\partial Y^*}\left[
    z_N(k) + \|X_N\|_T
    \right] + 2\Delta_\delta\hat{\omega}
\end{equation}
Thus, we can apply
the discrete Gronwall's inequality
to $y_N(k) \le z_N(k) + \hat{\sigma}_N(s_k)$,
and demonstrate that $y_N(T)$ is bounded universally
regardless of $N$.
By~(\ref{z.hat}) we conclude that
$\hat{Z}_N(s)$ is uniformly bounded and equicontinuous,
and by~(\ref{sigma.hat}) that $\hat{\sigma}_N(s)$ is
uniformly bounded.

We can now present
the upper bound
for $|\hat{\sigma}_N(s+\delta)-\hat{\sigma}_N(s)|$
when $\delta > 0$ is arbitrarily fixed.
If no update by (\ref{s.out}) is completed
between $s$ and $s+\delta$ then
\begin{math}
  |\hat{\sigma}_N(s+\delta)-\hat{\sigma}_N(s)|
  \le 2\Delta_\delta\hat{\omega} .
\end{math}
Otherwise, we can find the first and the last
update by (\ref{s.out}) completed
respectively at $s_k$ and $s_{k'}$ on $[s,s+\delta]$,
and observe that
\begin{align*}
  & |\hat{\sigma}_N(s+\delta)-\hat{\sigma}_N(s)|
  \le 8\Delta_\delta\hat{\omega}
  + |\hat{Z}_{1,N}(s_{k'})-\hat{Z}_{1,N}(s_{k})|
  \\ & \hspace{0.2in}
  + |\partial\hat{Y}^*_N(s_{k'},\hat{Z}_{(-1),N}(s_{k'}))
  - \partial\hat{Y}^*_N(s_{k},\hat{Z}_{(-1),N}(s_{k}))| .
\end{align*}
By using the construction of $X_N$ for the equicontinuity proof of
Lemma~\ref{h.equicont} we can find
\begin{align*}
  & |\partial\hat{Y}^*_N(s_{k'},\hat{Z}_{(-1),N}(s_{k'}))
  - \partial\hat{Y}^*_N(s_k,\hat{Z}_{(-1),N}(s_k))|
  \\ & \hspace{7.5ex}
  \le \sqrt{2}K_{\partial Y^*}\Delta_\delta X_N
  + K_{\partial Y^*}\Delta_\delta\hat{Z}_{(-1),N} .
\end{align*}
Thus, the equicontinuity of $\hat{Z}_N(s)$
implies that of $\hat{\sigma}_N(s)$.
\end{proof}

By Lemma~\ref{x.unif} and Proposition~\ref{y.tight}
we can also find that
$\hat{Y}_N(s)$ is uniformly bounded and equicontinuous.

\section{Skorohod equations}
\label{rel.skd}

Let $\kappa(s)$ be a real-valued continuous function.
Assuming $\kappa(0)\ge 0$,
we call
\begin{equation}\label{skd.eq}
  \eta(s) = \kappa(s) + \ell(s)
\end{equation}
a \emph{Skorohod equation} if $\eta(s)$ is a nonnegative continuous
function and $\ell(s)$ is a nondecreasing continuous function with
$\ell(0) = 0$, satisfying
\begin{equation}\label{skd.ell}
  \ell(s) = \int_0^s I_{\{\eta(u)=0\}}d\ell(u),
\end{equation}
where $I_{\{\eta(u)=0\}}$ is the indicator function of a statement
$\{\eta(u)=0\}$, taking values $1$ or $0$ accordingly as
the statement is true or not.
Given $\kappa(s)$, a pair $(\kappa,\ell)$ of functions forms the
Skorohod equation, and the nonnegative function $\eta$ of
(\ref{skd.eq}) is uniquely determined by
$\ell(s) = -\min_{0\le u\le s} [\kappa(u)\wedge 0]$.

\subsection{Backward Skorohod flow}\label{skd.def}

Let $T > 0$ be fixed.
As in the case of Section~\ref{y.cons.sec}
we consider an $\mathbb{F}_1$-valued process
started from $\hat{Y}^*(0) = \Phi_T^{-1}(y^*,\hat{\omega}')$,
and evolved backward in time by
$\hat{Y}^*(s) = \Phi_s(\hat{Y}^*(0),\hat{\omega}')$
until $\hat{Y}^*(T) = y^*$.
Provided $x\in \hat{Y}^*(0)$,
we can introduce equations of Skorohod type by
\begin{align}
  \label{skd}
  \hat{Y}^\dagger(s)
  & = x - \int_0^s\beta(\hat{Y}^\dagger(u))du
  + \hat{\omega}^{\hat{L}}(s);
  \\ \label{skd.l}
  \hat{L}(s)
  & = \int_0^s I_{\{
    \hat{Y}^\dagger_1(u)
    =\partial\hat{Y}^*(u,\hat{Y}^\dagger_{(-1)}(u))
    \}}
  d\hat{L}(u),
\end{align}
for $0\le s\le T$,
where $\hat{\omega}^{\hat{L}}$ is formed by
\begin{math}
  \hat{\omega}^{\hat{L}}(s)
  =[(\hat{\omega}_1-\hat{L})(s),\hat{\omega}_{(-1)}(s)].
\end{math}
And we set $\hat{L}(s) = 2\hat{\omega}_1(s)$
and $\hat{\omega}^{\hat{L}}(s) = \hat{\omega}'(s)$
for $0\le s\le T$
if $x\not\in\hat{Y}^*(0)$.
Consequently we can extend $\hat{\omega}^{\hat{L}}$
to a map
\begin{equation*}
  \hat{\Theta}_{x,y^*,T}(\hat{\omega}) =
  \begin{cases}
    \hat{\omega}^{\hat{L}}(s)
    & \mbox{ if $0\le s\le T$; } \\
    [\hat{\omega}_1(s)-\hat{L}(T),\hat{\omega}_{(-1)}(s)]
    & \mbox{ if $s > T$, }
  \end{cases}
\end{equation*}
from
\begin{math}
  (x,y^*,T,\hat{\omega})\in
  \mathbb{R}^n\times\mathbb{F}_1\times\mathbb{R}_+\times
  C(\mathbb{R}_+,\mathbb{R}^n)
\end{math}
to $C(\mathbb{R}_+,\mathbb{R}^n)$,
and call it a \emph{backward Skorohod flow}.
When restricted
from $\hat{\omega}\in C([0,T],\mathbb{R}^n)$
to $C([0,T],\mathbb{R}^n)$,
the backward Skorohod flow is simply denoted by
$\hat{\omega}^{\hat{L}}$.

Assuming a solution $(\hat{Y}^\dagger,\hat{L})$ to
(\ref{skd})--(\ref{skd.l}),
we can introduce
\begin{equation}\label{skd.z}
  \hat{Z}_1(s)
  = \hat{Y}^\dagger_1(s) + \hat{L}(s)
  = x_1
  - \int_0^s\beta_1(\hat{Y}^\dagger(u))du + \hat{\omega}_1(s).
\end{equation}
Similarly to Section~2 of Saisho and Tanemura~\cite{saisho},
one can set
\begin{math}
  \kappa(s)
  = \partial\hat{Y}^*(s,\hat{Y}^\dagger_{(-1)}(s)) - \hat{Z}_1(s)
\end{math}
and $\ell(s) = \hat{L}(s)$,
and show that
the pair $(v,\ell)$ satisfies
the Skorohod equations (\ref{skd.eq})--(\ref{skd.ell})
with
\begin{math}
  \eta(s)
  = \partial\hat{Y}^*(s,\hat{Y}^\dagger_{(-1)}(s)) - \hat{Y}^\dagger_1(s).
\end{math}
The uniqueness of solution to SDE of Skorohod type
is an immediate consequence of the following lemma.

\begin{lemma}\label{y.dag.cont}
Suppose that $x, x^\ddagger \in \hat{Y}^*(0)$,
and that $(\hat{Y}^\dagger,\hat{L})$ and $(\hat{Y}^\ddagger,\hat{L}^\ddagger)$
are solutions to (\ref{skd})--(\ref{skd.l})
with their respective initial states
$\hat{Y}^\dagger(0) = x$ and $\hat{Y}^\ddagger(0) = x^\ddagger$.
Then we have
\begin{equation*}
  \|\hat{Y}^\dagger-\hat{Y}^\ddagger\|_T
  \le (3+K_{\partial Y^*})\|x-x^\ddagger\|
  e^{(3+K_{\partial Y^*})K_\beta T}
\end{equation*}
\end{lemma}

\begin{proof}
Accompanying with the respective solutions
$(\hat{Y}^\dagger,\hat{L})$ and $(\hat{Y}^\ddagger,\hat{L}^\ddagger)$,
we can construct $\hat{Z}_1$ and $\hat{Z}_1^\ddagger$
by (\ref{skd.z}).
By the Lipschitz continuity of $\beta$, we have
\begin{equation*}
  |\hat{Z}_1-\hat{Z}_1^\ddagger|_s,\:
  \|\hat{Y}^\dagger_{(-1)}-\hat{Y}_{(-1)}^\ddagger\|_s
  \le \|x-x^\ddagger\|
  + K_\beta\int_0^s\|\hat{Y}^\dagger-\hat{Y}^\ddagger\|_u\,du
\end{equation*}
Applying Lemma~2.1 of~\cite{saisho}
and the Lipschitz constant $K_{\partial Y^*}$
for the process $\partial\hat{Y}^*(s,\cdot)$ of hypographical surface,
we can show that
\begin{align*}
  |\hat{L}-\hat{L}^\ddagger|_s
  & \le |\partial\hat{Y}^*(\cdot,\hat{Y}^\dagger_{(-1)}(\cdot))
  - \partial\hat{Y}^*(\cdot,\hat{Y}_{(-1)}^\ddagger(\cdot))|_s
  + |\hat{Z}_1-\hat{Z}_1^\ddagger|_s
  \\
  & \le K_{\partial Y^*} \|\hat{Y}^\dagger_{(-1)}-\hat{Y}_{(-1)}^\ddagger\|_s
  + |\hat{Z}_1-\hat{Z}_1^\ddagger|_s
\end{align*}
Together we obtain
\begin{equation*}
  \|\hat{Y}^\dagger-\hat{Y}^\ddagger\|_s
  \le (3+K_{\partial Y^*})\left[
    \|x-x^\ddagger\|
    + K_\beta \int_0^s\|\hat{Y}^\dagger-\hat{Y}^\ddagger\|_u\,du
    \right] ,
\end{equation*}
which completes the proof by Gronwall's inequality.
\end{proof}

The next proposition
establishes the existence of 
solution to SDE of Skorohod type.

\begin{proposition}\label{y.dag.conv}
Assume that $x_N\in\hat{Y}^*_N(0)$
converges to $x\in\hat{Y}^*(0)$.
Then $\hat{\sigma}_N$
of Algorithm~\ref{y.cons.alg}
uniformly converges to $\hat{L}$ of
(\ref{skd})--(\ref{skd.l}).
\end{proposition}

\begin{proof}
By Proposition~\ref{y.tight}
we can find a uniformly converging subsequence
for pairs $(\hat{Y}_{N_i},\hat{\sigma}_{N_i})$.
Clearly the limit $\hat{\sigma}$ of $\hat{\sigma}_{N_i}$
is nondecreasing, and
the limit $\hat{Y}$ satisfies
$\hat{Y}_1(s)\le\partial\hat{Y}^*(s,\hat{Y}_{(-1)}(s))$
for $0\le s\le T$.
Recall that
(\ref{y.at.k}) holds whenever
$\hat{\sigma}_{N_i}(s_k)-\hat{\sigma}_{N_i}(s_{k-1})>0$.
Furthermore,
$\partial\hat{Y}^*_{N_i}(s,\hat{Y}_{(-1),N_i}(s))$
converges uniformly to
$\partial\hat{Y}^*(s,\hat{Y}_{(-1)}(s))$
by Lemma~\ref{h.equicont}.
Thus, for arbitrary $\varepsilon_0 > 0$ we can find sufficiently
large $N_i$ so that
\begin{math}
  0\le\partial\hat{Y}^*(s_k,\hat{Y}_{(-1)}(s_k))
  -\hat{Y}_1(s_k)<\varepsilon_0
\end{math}
whenever
$\hat{\sigma}_{N_i}(s_k)-\hat{\sigma}_{N_i}(s_{k-1})>0$,
and
\begin{equation*}
  \hat{\sigma}_{N_i}(s_l)
  = \sum_{k=1}^{l}
    I_{\{
      0\le\partial\hat{Y}^*(s_k,\hat{Y}_{(-1)}(s_k))
      -\hat{Y}_{1}(s_k)<\varepsilon_0
    \}}
  \times(\hat{\sigma}_{N_i}(s_k) - \hat{\sigma}_{N_i}(s_{k-1}))
\end{equation*}
for $l=1,\ldots,N_i$.
In addition we can choose $N_i$ for arbitrary $\varepsilon_1 > 0$
such that
\begin{math}
  |\hat{\sigma}_{N_i}-\hat{\sigma}|_T
  < \left.\varepsilon_1\right/3T .
\end{math}
Therefore, we obtain
\begin{equation*}
  \left|\hat{\sigma}(s_l)
  - \sum_{k=1}^{l}
    I_{\{
      0\le\partial\hat{Y}^*(s_k,\hat{Y}_{(-1)}(s_k))
      -\hat{Y}_{1}(s_k)<\varepsilon_0
    \}}
    \times(\hat{\sigma}(s_k) - \hat{\sigma}(s_{k-1}))
    \right| < \varepsilon_1 ,
\end{equation*}
which implies that
\begin{equation*}
  \hat{\sigma}(s)
  = \int_0^s
  I_{\{0\le\partial\hat{Y}^*(u,\hat{Y}_{(-1)}(u))-\hat{Y}_1(u)<\varepsilon_0\}}
  \,d\hat{\sigma}(u) .
\end{equation*}
Since $\varepsilon_0 > 0$ is arbitrary,
the limit $\hat{\sigma}$ must satisfy
(\ref{skd.l}) with $\hat{L} = \hat{\sigma}$.
Similarly the limit $\hat{Y}$ satisfies (\ref{skd}).
By the uniqueness of solution
the whole sequence $\hat{Y}_N$ and $\hat{\sigma}_N$ must converge.
\end{proof}

We fix
\begin{math}
  (x,y^*,T)\in\mathbb{R}^n\times\mathbb{F}_1\times\mathbb{R}_+ ,
\end{math}
and construct a sequence $\hat{\sigma}_N$
of Algorithm~\ref{y.cons.alg} with
the initial state $\hat{Y}_N(0) = x$.
(i)
Provided $x\not\in\Phi_T^{-1}(y^*,\hat{\omega}')$,
we find $x\not\in\hat{Y}^*_N(0)$ for sufficiently large $N$,
and $\hat{\sigma}_N(s) = 2\hat{\omega}_1(s)$, $0\le s\le T$.
(ii)
Provided
$x\in\Phi_T^{-1}(y^*\setminus\partial y^*,\hat{\omega}')$,
we find $x\in\hat{Y}^*_N(0)$ for sufficiently large $N$.
By Proposition~\ref{y.dag.conv}
the whole sequence of $\hat{\sigma}_N$
converges uniformly to $\hat{L}$.
Thus, in either (i) or (ii)
$\hat{\sigma}_N$ converges uniformly to the backward Skorohod flow
$\hat{\omega}^{\hat{L}}$.
Since
\begin{math}
  \mathbb{W}\left(\left\{
  \hat{\omega}\in C([0,T],\mathbb{R}^n):
  x\in\Phi_T^{-1}(\partial y^*,\hat{\omega}')
  \right\}\right)=0,
\end{math}
the restriction of sample space on the cases of (i) and (ii)
does not change the result of Lemma~\ref{weak.w}.
Hence, 
we obtain the following corollary
to Proposition~\ref{y.dag.conv}.

\begin{corollary}\label{skd.flow.dist}
If a sample path $\hat{\omega}$ is distributed as
$\mathbb{W}$ on $C(\mathbb{R}_+,\mathbb{R}^n)$
then so is $\hat{\Theta}_{x,y^*,T}(\hat{\omega})$.
\end{corollary}

\subsection{Dynamical systems of Liggett dual}
\label{ligg.cons.sec}

For any fixed $T > 0$
we can define a map
\begin{equation*}
  \psi_{y^*,T}(x,\hat{\omega}) =
    \Phi_T(x,\hat{\Theta}_{x,y^*,T}(\hat{\omega}))
\end{equation*}
from
\begin{math}
  (x,y^*,\hat{\omega})\in
  \mathbb{R}^n\times\mathbb{F}_1\times
  C([0,T],\mathbb{R}^n)
\end{math}
to $\mathbb{R}^n$.
When $\hat{\omega}$ is fixed and
$x\in\Phi^{-1}_T(y^*,\hat{\omega}')$,
the map $\psi_{y^*,T}(x,\hat{\omega})$
gives a solution $\hat{Y}^\dagger(T)$
to (\ref{skd})--(\ref{skd.l}).
In terms of distribution
by Corollary~\ref{skd.flow.dist}
we can immediately observe

\begin{corollary}\label{stg.y.hat}
Let $\hat{\omega}$ be distributed as $\mathbb{W}$ over
$C([0,T],\mathbb{R}^n)$, and
let $P_T$ be the Markov transition kernel for~(\ref{c.sde}).
Then we have
\begin{equation*}
  \mathbf{E}_\mathbb{W}[f(\psi_{y^*,T}(x,\hat{\omega}))]
  = \int P_T(x,dy)f(y)
\end{equation*}
for any measurable function $f$ on $\mathbb{R}^n$.
\end{corollary}

In terms of consistency of dynamical system
we obtain the following lemma.

\begin{lemma}\label{psi.cont}
If $\psi_{y^*,t}(x,\omega(t-\cdot))$
is restricted on $x\in\Phi_t^{-1}(y^*,\omega'(t-\cdot))$
for each $t$,
then it is consistent.
\end{lemma}

\begin{proof}
Suppose that an increasing or a decreasing sequence $\{t_N\}$
converges to $t_0$, and that
$x_N\in\Phi_{t_N}^{-1}(y^*,\omega'(t_N-\cdot))$
converges to $x_0$.
Then we must have $x_0\in\Phi_{t_0}^{-1}(y^*,\omega'(t_0-\cdot))$
by Proposition~\ref{ss.cont}.
For each pair $(t_N,x_N)$ we can find the corresponding solution
$\hat{Y}^\dagger_N(s)$, $0\le s\le t_N$,
to~(\ref{skd})--(\ref{skd.l}) of Skorohod type
starting from $\hat{Y}^\dagger_N(0)=x_N$.
By Proposition~\ref{y.tight} the approximations are
uniformly bounded and equicontinuous,
and so is $\{\hat{Y}^\dagger_N(s)\}$ on each interval
$[0,t_N\wedge t_0]$.
Similarly to the proof of Lemma~\ref{ss.consist}
we can argue that $\hat{Y}^\dagger_N(s)$ converges
uniformly on $[0,t_0)$ to
the solution $\hat{Y}^\dagger(s)$, $0\le s\le t_0$,
to~(\ref{skd})--(\ref{skd.l})
starting from $\hat{Y}^\dagger(0)=x_0$.
\end{proof}

It should be noted that
$\psi_{y^*,t}(x,\omega(t-\cdot))$
cannot be consistent for the unrestricted domain.
When
$x_N\not\in\Phi_t^{-1}(y^*,\omega'(t-\cdot))$
converges to $x_0\in\Phi_t^{-1}(y^*,\omega'(t-\cdot))$,
it is almost likely observed that
$\psi_{y^*,t}(x_0,\omega(t-\cdot))\in y^*\setminus\partial y^*$
while
$\lim_{N\to\infty}\psi_{y^*,t}(x_N,\omega(t-\cdot))\in \partial y^*$;
thus, the consistency fails.

\begin{proposition}\label{ligg.cons}
Let $\mathbb{F}_2$ be a subclass of closed subsets in
$\mathbb{R}^n$,
and let
\begin{equation*}
  D^* = \left\{
  (z^*,y^*)\in\mathbb{F}_2\times\mathbb{F}_1:
  z^*\subsetneq y^*
  \right\} .
\end{equation*}
Assuming
$Z^*(t) = \psi_{y^*,t}^{-1}(z^*,\omega(t-\cdot))$
is a lower semicontinuous $\mathbb{F}_2$-valued process
for each pair $(z^*,y^*)\in D^*$,
the Markov dynamical system
\begin{equation*}
  \Xi^*_t((z^*,y^*),\omega)
  =(\psi_{y^*,t}^{-1}(z^*,\omega(t-\cdot)),
  \Phi_t^{-1}(y^*,\omega'(t-\cdot)))
\end{equation*}
is a Liggett dual of (\ref{c.sde}) with respect to
\begin{equation}\label{f1.gamma}
  \Gamma((z^*,y^*),x) = \begin{cases}
    1 & \mbox{ if $x\in y^*\setminus z^*$; } \\
    0 & \mbox{ otherwise. }
  \end{cases}
\end{equation}
\end{proposition}

\begin{proof}
Let $x\in D$ and $(z^*,y^*)\in D^*$ be arbitrarily fixed.
Then $\psi_{y^*,t}$ clearly satisfies
\begin{equation*}
  \Gamma((z^*,y^*),\psi_{y^*,t}(x,\omega(t-\cdot)))
  = \Gamma(\Xi^*_t((z^*,y^*),\omega),x) .
\end{equation*}
By Corollary~\ref{stg.y.hat} we obtain
\begin{align*}
  & \int P_t(x,dy)\Gamma((z^*,y^*),y)
  \\ & \hspace{2.5ex}
  = \mathbf{E}_{\mathbb{W}}[
    \Gamma((z^*,y^*),\psi_{y^*,t}(x,\omega(t-\cdot)))]
  = \mathbf{E}_{\mathbb{W}}[
    \Gamma(\Xi^*_t((z^*,y^*),\omega),x)] ,
\end{align*}
as desired.
\end{proof}

Proposition~\ref{ligg.cons}
generates a pair $(Z^*(t),Y^*(t))$ of closed sets
on the state space $D^*$
until the absorbing time
$\zeta = \inf\{t\ge 0: Z^*(t)=Y^*(t)\}$.
When restricted as in Lemma~\ref{psi.cont}
the Skorohod map $\psi_{y^*,t}(\cdot,\omega(t-\cdot))$
is continuous but not bijective;
thus, $Z^*(t) = \psi_{y^*,t}^{-1}(x^*,\omega(t-\cdot))$
could be absorbed into the empty set $\varnothing$,
a coffin state of $\mathbb{F}_2$.
For all the examples of Section~\ref{m1.sec}
we can set
$\mathbb{F}_2 = \mathbb{F}_1$ in Proposition~\ref{ligg.cons},
and observe that
\begin{math}
  Z^*(t) = \Phi_t^{-1}(z^*,\omega(t-\cdot))
\end{math}
until the absorbing time $\zeta$.
Therefore,
the Markov dynamical system
\begin{equation*}
  \Xi^*_t((z^*,y^*),\omega)
  =(\Phi_t^{-1}(z^*,\omega(t-\cdot)),
  \Phi_t^{-1}(y^*,\omega'(t-\cdot)))
\end{equation*}
becomes a Liggett dual
in the following examples.

\begin{example}\label{r1.ex.ligg}
In Example~\ref{r1.ex} the hypographical surface
$\partial Z^*(t,\cdot)\equiv Z_1(t)$
is determined by
$dZ_1(t) = \beta_1(Z_1(t))dt + dW_1(t)$.
Therefore, a Liggett dual of Proposition~\ref{ligg.cons}
is formed by a stochastic process $(Z_1(t),Y_1(t))$
with respect to (\ref{bm.gamma})
on the dual state space~(\ref{bm.d}).
The two SDE's of $(Z_1(t),Y_1(t))$
correspond to the differential operator
\begin{equation}\label{r1.b}
  \mathcal{B}f = \left(
  \beta_1(y)\frac{\partial}{\partial y}
  + \beta_1(z)\frac{\partial}{\partial z}
  \right)\! f
  + \frac{1}{2}\left(
  \frac{\partial}{\partial y}
  - \frac{\partial}{\partial z}
  \right)^2\!\!f .
\end{equation}
with $f(z,y)$ tending to zero as
$(z,y)$ approaches the boundary
$\{(y,y)\in\mathbb{R}^2:y\in\mathbb{R}\}$.
Hence, the Liggett dual of Theorem~\ref{bm.cons}
is viewed as a special case of Proposition~\ref{ligg.cons}.
\end{example}

\begin{example}\label{r2.ex.ligg}
In Example~\ref{r2.ex}
the pair
\begin{math}
  \partial Y^*(t,x_2)
  = (U_1(t)/U_2(t))(x_2-Y_2(t))+Y_1(t)
\end{math}
and
\begin{math}
  \partial Z^*(t,x_2)
  = (U_1(t)/U_2(t))(x_2-Z_2(t))+Z_1(t)
\end{math}
of hypographical surfaces share the common direction
determined by $dU(t) = \beta(U(t))dt$.
Thus, the Liggett dual is formulated by
the triplet $(U(t),Z(t),Y(t))$
of $\mathbb{R}^2$-valued processes
on a dual state space
\begin{equation*}
  D^* = \big\{
  (u,z,y)\in\mathbb{R}^6:
  \langle[u_2,-u_1]^T,y-z\rangle > 0,\,|u_1|< u_2
  \big\}.
\end{equation*}
We can set a duality function
$\Gamma((u,z,y),x) = 1$
if $\langle[u_2,-u_1]^T,y-x\rangle \ge 0$ 
and $\langle[u_2,-u_1]^T,x-z\rangle > 0$;
otherwise, $\Gamma((u,z,y),x) = 0$.
Here the governing SDE's correspond to
the differential operator
\begin{align*}
  \mathcal{B}f
  & = \left(
  u_2\frac{\partial}{\partial u_1}
  + u_1\frac{\partial}{\partial u_2}
  + y_2\frac{\partial}{\partial y_1}
  + y_1\frac{\partial}{\partial y_2}
  + z_2\frac{\partial}{\partial z_1}
  + z_1\frac{\partial}{\partial z_2}
  \right)\!f
  \\ & \hspace{5.5ex}
  + \frac{1}{2}\left(
  \frac{\partial}{\partial y_1}
  - \frac{\partial}{\partial z_1}
  \right)^2\!\!f
  + \frac{1}{2}\left(
  \frac{\partial}{\partial y_2}
  + \frac{\partial}{\partial z_2}
  \right)^2\!\!f
\end{align*}
with $f(u,z,y)$ tending to zero as
$(u,z,y)$ approaches the boundary
of
\begin{math}
  \bar{D}^* = \{(u,z,y)\in\mathbb{R}^6:
  \langle[u_2,-u_1]^T,y-z\rangle \ge 0,
  \,|u_1|< u_2\}.
\end{math}
\end{example}

\begin{example}\label{b.model.ligg}
We set $\theta=\pi/2$ in Example~\ref{b.model}.
Then the pair $\partial Y^*(t) = H + Y(t)$
and $\partial Z^*(t) = H + Z(t)$ of hypographical surfaces is
determined by $(Z(t),Y(t))$ on a dual state space
\begin{equation*}
  D^* = \{(z,y)\in\mathbb{R}^{2n}:
  \langle d,y-z\rangle > 0\} .
\end{equation*}
We can introduce a duality function
$\Gamma((z,y),x) = 1$ if
$\langle d,y-x\rangle \ge 0$
and $\langle d,x-z\rangle > 0$;
otherwise, $\Gamma((z,y),x) = 0$.
Then $(Z(t),Y(t))$ provides a Liggett dual,
and it is governed by
\begin{align*}
  \mathcal{B}f
  & = \sum_{i=1}^n\left(
  \beta_i(y)\frac{\partial}{\partial y_i}
  + \beta_i(z)\frac{\partial}{\partial z_i}
  \right)\!f
  \\ & \hspace{5.5ex}
  + \frac{1}{2}\left(
  \frac{\partial}{\partial y_1}
  - \frac{\partial}{\partial z_1}
  \right)^2\!\!f
  + \frac{1}{2}\sum_{i=2}^n\left(
  \frac{\partial}{\partial y_i}
  + \frac{\partial}{\partial z_i}
  \right)^2\!\!f
\end{align*}
with $f(z,y)$ vanishing as $(z,y)$ approaches
the boundary $\partial D^*$.
\end{example}

In Example~\ref{b.model.ligg} we may choose
$\theta_0<\theta<\pi/2$ for $\mathbb{F}_1$,
and set an initial hyperplane $\partial Y^*(0) = H_1\in\mathbb{F}_1$
not parallel to $H$.
Proposition~\ref{ligg.cons} is applicable
by setting $\mathbb{F}_2 = \mathbb{F}_0\cup\{\varnothing\}$,
but the exact sample path $(Z^*(t),Y^*(t))$ of Liggett dual
is no longer tractable.

\section{Forward Skorohod flow}
\label{impute.sec}

Let $T>0$ be fixed.
Similarly to Section~\ref{skd.def}
we consider a forward process
\begin{equation}\label{sde.fwd}
  X(t) = \Phi_t^{-1}(X(0),\omega(t-\cdot)),
  \quad 0\le t\le T,
\end{equation}
starting from $X(0) = \Phi_T(x,\omega(T-\cdot))$
so that it terminates at $X(T) = x$.
Provided $X(0)\in y^*$,
we can determine a sample path
\begin{equation}\label{skd.w.fwd}
  \omega^L(t) = [(\omega_1-L)(t),\omega_{(-1)}(t)]
\end{equation}
by forming an increasing process
\begin{equation}\label{skd.l.fwd}
  L(t)
  = \int_0^t I_{\{X(v)\in\partial\Phi_v^{-1}(y^*,(\omega^L)'(v-\cdot))\}}
  dL(v),
\end{equation}
so that $\omega^L$ satisfies
\begin{equation*}
  X(t) \in \Phi_t^{-1}(y^*,(\omega^L)'(t-\cdot))
\end{equation*}
for $0\le t\le T$.
By setting $L(t) = 2\omega_1(t)$
and $\omega^{L}(t) = \omega'(t)$
for $0\le t\le T$
if $X(0)\not\in y^*$,
we can extend $\omega^{L}$ to a map
\begin{equation*}
  \Theta_{x,y^*,T}(\omega) =
  \begin{cases}
    \omega^{L}(t)
    & \mbox{ if $0\le t\le T$; } \\
    [\omega_1(t)-L(T),\omega_{(-1)}(t)]
    & \mbox{ if $t > T$, }
  \end{cases}
\end{equation*}
from
\begin{math}
  (x,y^*,T,\hat{\omega})\in
  \mathbb{R}^n\times\mathbb{F}_1\times\mathbb{R}_+\times
  C(\mathbb{R}_+,\mathbb{R}^n)
\end{math}
to $C(\mathbb{R}_+,\mathbb{R}^n)$.

Assuming a backward sample path $\hat{\xi}(s)$, $0\le s\le T$,
and $x\in\Phi^{-1}_T\big(y^*,\hat{\xi}'\big)$,
we can find a backward Skorohod flow $\hat{\xi}^{\hat{L}}$
to (\ref{skd})--(\ref{skd.l}).
Then a solution to (\ref{skd.w.fwd})--(\ref{skd.l.fwd}) exists
if the time-reversed process of (\ref{sde.fwd}) is
$\hat{X}(s) = \Phi_s\big(x,\hat{\xi}^{\hat{L}}\big)$,
$0\le s\le T$.
In fact, we set
$\omega(t) = \hat{\xi}^{\hat{L}}(T-t) - \hat{\xi}^{\hat{L}}(T)$,
and obtain
the solution of $L(t) = \hat{L}(T)-\hat{L}(T-t)$
and $\omega^L(t) = \hat{\xi}(T-t)-\hat{\xi}(T)$.
Thus, we can appropriately call
$\omega^L$ a \emph{forward Skorohod flow}.

\subsection{A coupled approximation of forward process}
\label{s.ligg.sec}

In order to construct a forward Skorohod flow
by approximation,
we take Algorithm~\ref{y.cons.alg} and run steps forward in time.
Here we start with an entire path $X_N(t)$ of approximation by
(\ref{x.fwd}) with sample path $\omega$,
in which $X_N(0)$ converges to $X(0)$.
Then we build $\sigma_N(t)$ forward,
and generate the sample path
\begin{equation*}
  \omega^{\sigma_N}(t)
  = [(\omega_1-\sigma_N)(t),\omega_{(-1)}(t)] ,
\end{equation*}
which is used to approximate $Y^*_N(t)$ recursively.

\begin{algorithm}\label{s.ligg.alg}
Set the initial values
\begin{equation*}
  Y^*_N(0) = y^*;
  \quad
  U_N(0) = \partial y^*(X_{(-1),N}(0)),
\end{equation*}
and $\sigma_N(0)=0$ at $t_0 = 0$.
At the $j$-th step,
provided $U_N(t_{j-1})$ and $\sigma_N(t_{j-1})$,
(i) set for $t_{j-1}<t\le t_j$
\begin{equation}\label{fwd.s.out}
  \sigma_N(t) = \sigma_N(t_{j-1})
  + 2(\omega_1(t)-\omega_1(t_{j-1}))
\end{equation}
if
\begin{equation*}
  X_{1,N}(t_j) - \beta_1(X_N(t_j))(t_j-t_{j-1})
  + |\omega_1(t_j) - \omega_1(t_{j-1})|
  > U_N(t_{j-1})
\end{equation*}
otherwise, set
\begin{equation}\label{fwd.s.in}
  \sigma_N(t) \equiv \sigma_N(t_{j-1}) .
\end{equation}
(ii) Update
\begin{align*}
  \nonumber
  Y^*_N(t_j)
  & = \phi_{t_j-t_{j-1}}^{-1}(Y^*_N(t_{j-1}),
  (\omega^{\sigma_N})'(t_j-\cdot));
  \\
  U_N(t_j)
  & = \partial Y^*(t_j,X_{(-1),N}(t_j))
\end{align*}
at $t = t_j$.
\end{algorithm}

\begin{remark}\label{s.ligg.rem}
We generate $\hat{X}_N(s)$ by (\ref{x.euler}),
$0\le s\le T$,
and view the time-reversed $X_N(t) = \hat{X}_N(T-t)$
as if it were constructed by (\ref{x.fwd}),
for which we set
$\hat{\omega}(s) = \omega(T-s) - \omega(T)$,
$0\le s\le T$.
The update by (\ref{fwd.s.out}) in Algorithm~\ref{s.ligg.alg}
is determined by $X_N(t_j)$, $X_{(-1),N}(t_{j-1})$,
$|\omega_1(t_j)-\omega_1(t_{j-1})|$, and $\partial Y^*(t_{j-1})$,
which is generated by the sample path $\hat{\omega}(s)$,
$0\le s\le T-t_{j-1}$
and $\omega^{\sigma_N}(t)$, $0\le t\le t_{j-1}$.
In the context of Algorithm~\ref{y.cons.alg}
we can view $\partial Y^*_N(t_{j-1})$
as if it were generated by sample path $\xi = \omega^{\sigma_N}$,
and $\hat{X}_N(T-t_{j-1}) = X_N(t_{j-1})$
as if it were updated backward
with $\hat{\xi}^{\hat{\sigma}_N}(s)$, $0\le s\le T-t_{j-1}$,
for which we set
$\hat{\sigma}_N(s) = \sigma_N(T) - \sigma_N(T-s)$.
Hence a version of Lemma~\ref{weak.w} can be argued
for $\omega^{\sigma_N}(t)$,
and the following lemma is similarly established.
\end{remark}

\begin{lemma}\label{weak.fwd.w}
Let $X_N(t)$ be the time-reversed one to
$\hat{X}_N(s)$ of Remark~\ref{s.ligg.rem}.
If $\omega$
is distributed as $\mathbb{W}$ over $C([0,T],\mathbb{R}^n)$
then so is
$\omega^{\sigma_N}$ of Algorithm~\ref{s.ligg.alg}.
\end{lemma}

In addition to the construction of $Y^*_N(t_j)$ in
Algorithm~\ref{s.ligg.alg}
we can introduce a series of estimates $Y_N(t_j)$
in an attempt to predict the point $[U_N(t_j),X_{(-1),N}(t_j)]$
for $j=1,\ldots,N$.

\begin{algorithm}\label{u.traj.alg}
If $j=1$ or $\sigma_N(t)$ on the $j$-th interval $(t_{j-1},t_j]$
is updated by~(\ref{fwd.s.out}) then restart
\begin{equation*}
  Y_N(t_j) = \phi_{t_j-t_{j-1}}^{-1}\left(
  [U_N(t_{j-1}),X_{(-1),N}(t_{j-1})],
  (\omega^{\sigma_N})'(\cdot+t_{j-1})
  \right)
\end{equation*}
with the exact point $[U_N(t_{j-1}),X_{(-1),N}(t_{j-1})]$.
Otherwise [i.e., $\sigma_N(t_j)$ is updated
by~(\ref{fwd.s.in})], set
\begin{equation*}
  Y_N(t_j) = \phi_{t_j-t_{j-1}}^{-1}\left(
  Y_N(t_{j-1}),(\omega^{\sigma_N})'(\cdot+t_{j-1})
  \right) .
\end{equation*}
\end{algorithm}

In what follows we assume a Lipschitz constant
$K_{\partial Y^*}\ge 1$
for $\partial Y_N^*(t_j)$ regardless of $t_j$ and $N$.
Then we can evaluate the proximity of $Y_{1,N}(t_j)$ to the
height $U_N(t_j)$ of the surface
$\partial Y_N^*(t_j)$ at $X_{(-1),N}(t_j)$.

\begin{lemma}\label{u.approx}
Assume that there is the last update of $\sigma_N(t_j)$
by~(\ref{fwd.s.out}) at the interval $(t_{j-1},t_j]$
before $t_k$ in Algorithm~\ref{s.ligg.alg}.
Then we have
\begin{align*}
  \nonumber &
  |Y_{1,N}(t_k) - U_N(t_k)|
  \\ & \hspace{2.5ex}
  \le K_{\partial Y^*}K_\beta\Big(
  \max_{0\le i\le N}\|Y_N(t_i)\|
  + \max_{0\le i\le N}\|X_N(t_i)\|
  \Big)(t_k-t_{j-1})
\end{align*}
\end{lemma}

\begin{proof}
Similarly to (\ref{x.sum.fwd})
we can formulate $Y_N(t_k)$ implicitly by
\begin{align*}
  Y_N(t_k) & = [U_N(t_{j-1}),X_{(-1),N}(t_{j-1})]
  \\ &
  + \sum_{i=j}^k \beta(Y_N(t_i))(t_i-t_{i-1})
  + (\omega^{\sigma_N})'(t_k) - (\omega^{\sigma_N})'(t_{j-1}) .
\end{align*}
Since $Y_{N}(t_k)\in \partial Y^*_N(t_k)$,
$|Y_{1,N}(t_k) - U_N(t_k)|$ is bounded by
\begin{align*}
  & K_{\partial Y^*}\left\|Y_{(-1),N}(t_k)-X_{(-1),N}(t_k)\right\|
  \\ & \hspace{2.5ex}
  \le K_{\partial Y^*}\Big\|
  \sum_{i=j}^k\beta_{(-1)}(Y_N(t_i))(t_i-t_{i-1})
  - \sum_{i=j}^k\beta_{(-1)}(X_N(t_i))(t_i-t_{i-1})
  \Big\|
\end{align*}
which is further bounded by the one as desired.
\end{proof}

Under the assumption of Lemma~\ref{u.approx}
we can observe that $\|X_N(t_i)\|$ and $\|Y_N(t_i)\|$
are not far apart for $i=j,\ldots,k$.
At the update by (\ref{fwd.s.out})
we have
\begin{math}
  U_N(t_{j-1})
  \le X_{1,N}(t_{j-1}) + 2(\omega_1(t_j)-\omega_1(t_{j-1})) .
\end{math}
Since
\begin{math}
  (\omega^{\sigma_N})'(t_i)-(\omega^{\sigma_N})'(t_j)
  = \omega'(t_i)-\omega'(t_j)
\end{math}
for $i = j,\ldots,k$,
$X_N(t_i)$ and $Y_N(t_i)$
are similarly updated.
In particular,
by Lemma~\ref{x.unif.fwd} we can find an upper bound for
$\max_{0\le i\le N}\|X_N(t_i)\|$
and $\max_{0\le i\le N}\|Y_N(t_i)\|$
regardless of $N$.

\begin{proposition}\label{ell.tight}
$\sigma_N(t)$ is equicontinuous on $[0,T]$.
\end{proposition}

\begin{proof}
If $X(0)\not\in y^*$ then $\sigma_N(t) = 2\omega_1(t)$
is equicontinuous;
thus, we assume $X(0)\in y^*$.
Let $\delta > 0$ and $0\le t< t'\le T$
be fixed such that $t'-t\le\delta$.
Clearly we have
$|\sigma_N(t')-\sigma_N(t)|\le 4\Delta_\delta\omega$
if there is no complete update by (\ref{fwd.s.out})
over the interval $(t,t']$; otherwise,
we find a series of updates by (\ref{fwd.s.out}),
say the first one on $(t_{k_1-1},t_{k_1}]$
to the last one on $(t_{k_\ell-1},t_{k_\ell}]$
between $t$ and $t'$.
Then $|\sigma_N(t)-\sigma_N(t')|$ is bounded by
\begin{align*}
  & |\omega_1^{\sigma_N}(t_{k_1-1}) - \omega_1^{\sigma_N}(t_{k_\ell-1})|
  + 7\Delta_\delta\omega_1
  \\ &
  \le \left|U_{1,N}(t_{k_1-1}) - U_{1,N}(t_{k_\ell-1})
  - \omega_1^{\sigma_N}(t_{k_\ell-1}) + \omega_1^{\sigma_N}(t_{k_1-1})\right|
  \\ & \hspace{10ex}
  + |X_{1,N}(t_{k_1-1}) - X_{1,N}(t_{k_\ell-1})| + 11\Delta_\delta\omega_1
\end{align*}
We can bound the first term of the upper bound above by
\begin{equation*}
  \sum_{i=2}^\ell\left|Y_{1,N}(t_{k_i-1}) - U_{1,N}(t_{k_i-1})\right|
  + \Big|\sum_{j=k_1}^{k_\ell-1}\beta_1(Y_N(t_j))(t_j-t_{j-1})\Big|
\end{equation*}
By Lemma~\ref{u.approx}
the above summation is further bounded by $C\delta$
with some constant value $C$ regardless of $N$.
Hence, we obtain
\begin{math}
  \Delta_\delta\sigma_N\le C\delta+\Delta_\delta X_{1,N}
  + 11\Delta_\delta\omega_1 .
\end{math}
\end{proof}

\subsection{Uniqueness and existence of forward Skorohod flow}

Let $T > 0$ and $y^*\in\mathbb{F}_1$ be fixed.
In order to show the uniqueness of forward Skorohod flow,
we consider two sample paths $\omega$ and $\xi$,
and generate two processes
$X(t)$ and $Y(t)$ of (\ref{sde.fwd})
respectively starting from 
$X(0), Y(0)\in y^*$.
In the next two lemmas we assume
the existence of their respective solutions
$(L, \omega^L)$ and $(M, \xi^M)$
to (\ref{skd.w.fwd})--(\ref{skd.l.fwd}),
and set the respective processes
\begin{math}
  Y^*(t) = \Phi_t^{-1}(y^*,(\omega^L)'(t-\cdot))
\end{math}
and
\begin{math}
  Z^*(t) = \Phi_t^{-1}(y^*,(\xi^M)'(t-\cdot))
\end{math}
so that
$X(t)\in Y^*(t)$
and $Y(t)\in Z^*(t)$ for $0\le t\le T$.

\begin{lemma}\label{bounded.gamma}
For $0\le t\le T$
a distance
\begin{equation*}
  d(t)
  = \sup\{\|\Phi_t^{-1}(z,(\omega^L)'(t-\cdot))
  - \Phi_t^{-1}(z,(\xi^M)'(t-\cdot))\|
  :z\in\partial y^*\}
\end{equation*}
between $\partial Y^*(t)$ and $\partial Z^*(t)$
is bounded by $\|\omega^L-\xi^M\|_T e^{K_\beta T}$.
\end{lemma}

\begin{proof}
We choose $z\in \partial y^*|_{D}$ arbitrarily,
and set $U(t)=\Phi_t^{-1}(z,(\omega^L)'(t-\cdot))$
and $V(t)=\Phi_t^{-1}(z,(\xi^M)'(t-\cdot))$.
Since
\begin{align*}
  & \|U(t)-V(t)\|
  \\ & \hspace{0.2in}
  = \left\|
  \int_0^t[\beta(U(v))-\beta(V(v))]dv
  + \omega^L(t) - \xi^M(t)
  \right\|
  \\ & \hspace{0.2in}
  \le \|\omega^L - \xi^M\|_T
  + K_\beta\int_0^t\|U(v)-V(v)\|dv,
\end{align*}
we obtain the upper bound by Gronwall's inequality.
\end{proof}

In the following proposition
we assume a Lipschitz constant
$K_{\partial Y^*}\ge 1$
for $\partial Y^*(t,\cdot)$
and $\partial Z^*(t,\cdot)$.
Furthermore, we set
$\gamma(t) = \sup_{0\le v\le t} d(v)$
and
\begin{align*}
  \theta(t)
  & = \sup\{\|\Phi_v^{-1}(z,(\omega^L)'(v-\cdot))
  - \Phi_v^{-1}(z,(\xi^M)'(v-\cdot))-n(v)+m(v)\|
  \\
  & \hspace{12.5ex}
  :z\in\partial y^*,\,0\le v\le t\} ,
\end{align*}
where
$n(t) = [L(t);0]$ and $m(t) = [M(t);0]$
are the $n$-dimensional vectors
at the direction of the first coordinate
having the respective length $L(t)$ and $M(t)$.

\begin{lemma}\label{skd.ell.d}
We have
\begin{align*}
  \gamma(T)
  & \le
  (1+\sqrt{2} K_{\partial Y^*})(3+K_\beta T)\|X-Y\|_T
  e^{(1+\sqrt{2} K_{\partial Y^*})K_\beta T} ;
  \\
  \theta(T)
  & \le (3+K_\beta T)\|X-Y\|_T + K_\beta T\gamma(T) ;
  \\
  |L-M|_T
  & \le \sqrt{2} K_{\partial Y^*}\left(
  \theta(T) + \|X-Y\|_T
  \right) .
\end{align*}
\end{lemma}

\begin{proof}
In the setting of proof of Lemma~\ref{bounded.gamma}
we can observe that
\begin{align*}
  & \|U(t)-V(t)-n(t)+m(t)\|
  \\ & \hspace{0.2in}
  = \left\|
  \int_0^t[\beta(U(v))-\beta(V(v))]dv
  + \omega'(t) - \xi'(t)
  \right\|
  \\ & \hspace{0.2in}
  \le \|\omega - \xi\|_T
  + K_\beta\int_0^t\|U(v)-V(v)\|dv
  \le \|\omega - \xi\|_T
  + K_\beta\int_0^t\gamma(v)dv .
\end{align*}
Since
$\|\omega - \xi\|_T \le (2 + K_\beta T)\|X - Y\|_T$,
we obtain
\begin{equation*}
  \theta(t) \le (2 + K_\beta T)\|X - Y\|_T
  + K_\beta\int_0^t\gamma(v)dv .
\end{equation*}

Observe that
\begin{math}
  \kappa(t) = \partial Y^*(t,X_{(-1)}(t)) - L(t) - X_1(t)
\end{math}
and $\ell(t) = L(t)$ form a Skorohod equation.
By applying Lemma~2.1 of~\cite{saisho}
we can show that
\begin{align*}
  & |L-M|_t \le\sup_{0\le v\le t}
  |\partial Y^*(v,X_{(-1)}(v)) -\partial Z^*(v,Y_{(-1)}(v))
  \\ & \hspace{1.75in}
  -L(v)+M(v)-X_1(v)+Y_1(v)|
\end{align*}
We can choose $z\in \partial y^*$ satisfying
\begin{math}
  [\partial Y^*(v,X_{(-1)}(v)),X_{(-1)}(v)]
  = \Phi_v^{-1}(z,(\omega^L)'(v-\cdot)) ,
\end{math}
and set
$U(v) = \Phi_v^{-1}(z,(\omega^L)'(v-\cdot))$ and
$V(v) = \Phi_v^{-1}(z,(\xi^M)'(v-\cdot))$.
Observing that $U_{(-1)}(v) = X_{(-1)}(v)$,
we obtain
\begin{align*}
  & |\partial Y^*(v,X_{(-1)}(v)) -\partial Z^*(v,Y_{(-1)}(v))
  -L(v)+M(v)-X_1(v)+Y_1(v)|
  \\ & \hspace{0.2in}
  \le |U_1(v)-V_1(v)-L(v)+M(v)|
  \\ & \hspace{0.55in}
  + K_{\partial Y^*}\|V_{(-1)}(v)-Y_{(-1)}(v)\|
  + |X_1(v)-Y_1(v)|
  \\ & \hspace{0.2in}
  \le \sqrt{2} K_{\partial Y^*}\left(
  \|U(v)-V(v)-n(v)+m(v)\| + \|X(v)-Y(v)\|
  \right) ,
\end{align*}
which implies that
\begin{math}
  |L-M|_t
  \le \sqrt{2} K_{\partial Y^*}\left(
  \theta(t) + \|X-Y\|_T
  \right) .
\end{math}
Together we can show that
\begin{align*}
  \gamma(t)
  & \le (1+\sqrt{2} K_{\partial Y^*})\theta(t)
  + \sqrt{2} K_{\partial Y^*}\|X-Y\|_T
  \\ &
  \le (1+\sqrt{2} K_{\partial Y^*})(3 + K_\beta T)\|X-Y\|_T
  + (1+\sqrt{2} K_{\partial Y^*})K_\beta\int_0^t\gamma(v)dv,
\end{align*}
which implies the upper bound for $\gamma(t)$.
\end{proof}

By Lemma~\ref{skd.ell.d} we find the uniqueness of
forward Skorohod flow,
and along with Proposition~\ref{ell.tight} we are ready for
the existence of such a sample path.
The proof requires a version of Lemma~\ref{h.equicont}
for $\partial Y^*_N(t)$ of Algorithm~\ref{s.ligg.alg}
with sample path $\omega^{\sigma_N}$
in order to show
the uniform convergence of
a subsequence $\partial Y^*_{N_i}(t,X_{(-1),N_i}(t))$.
Otherwise, the proof of Proposition~\ref{s.conv}
goes exactly as in Proposition~\ref{y.dag.conv}.

\begin{proposition}\label{s.conv}
Assuming that $X_N(0)\in y^*$ is convergent,
$\sigma_N$ of Algorithm~\ref{s.ligg.alg} uniformly converges to $L$
of (\ref{skd.w.fwd})--(\ref{skd.l.fwd}).
\end{proposition}

The approximation $X_N(t)$ for Algorithm~\ref{s.ligg.alg}
can be constructed by Remark~\ref{s.ligg.rem}.
(i)
If $X(0)\not\in y^*$
then we can find $X_N(0)\not\in y^*$
for sufficiently large $N$, and therefore, obtain
$\sigma_N(t) = 2\omega_1(t)$, $0\le t\le T$.
(ii)
If $X(0)\in y^*\setminus\partial y^*$
then
$X_N(0)\in y^*\setminus\partial y^*$ holds
for sufficiently large $N$, and therefore,
Proposition~\ref{s.conv} is applicable for $\sigma_N(t)$.
Hence, in either (i) or (ii)
$\omega^{\sigma_N}(t)$, $0\le t\le T$, converges uniformly
to $\Theta_{x,y^*,T}(\omega)$.
Since
\begin{math}
  \mathbb{W}\left(\left\{
  \omega\in C([0,T],\mathbb{R}^n):
  \Phi_T(x,\omega(T-\cdot))\in\partial y^*
  \right\}\right)=0,
\end{math}
the restriction of the sample space
on (i)--(ii) does not change the result of Lemma~\ref{weak.fwd.w}.
Thus, we obtain the following corollary.

\begin{corollary}\label{c.skd.dist}
If a sample path $\omega$ is distributed as
$\mathbb{W}$ on $C(\mathbb{R}_+,\mathbb{R}^n)$
then so is $\Theta_{x,y^*,T}(\omega)$.
\end{corollary}

\subsection{$\Lambda$-linked dynamical systems}

Let $T>0$ be fixed.
Provided $X\in C([0,T],\mathbb{R}^n)$,
we can impute $\omega\in C([0,T],\mathbb{R}^n)$ by setting
\begin{equation}\label{impute.w}
  \omega(v) = X(v) - X(0) - \int_0^v\beta(X(u))\,du,
  \quad 0\le v\le T,
\end{equation}
so that $X(v) = \Phi^{-1}_v(X(0),\omega(v-\cdot))$
for $0\le v\le t$.
Let $y^*\in\mathbb{F}_1$.
For any fixed $T > 0$ we can define a map
\begin{equation*}
  \tilde{\Theta}_{y^*,T}(X) =
  \Theta_{X(T),y^*,T}(\omega)
\end{equation*}
from
\begin{math}
  (y^*,X)\in
  \mathbb{F}_1\times C([0,T],\mathbb{R}^n)
\end{math}
to $C([0,T],\mathbb{R}^n)$.
The following proposition
is almost a restatement of Corollary~\ref{c.skd.dist},
and it provides a complete claim for
what we started in Proposition~\ref{pitman.cons}.

\begin{proposition}\label{y.hat.dist}
Construct $\hat{X}(s) = \Phi_s(x,\omega(T-\cdot))$,
$0\le s\le T$, using a sample path $\omega$
distributed as $\mathbb{W}$,
and impute $\omega$ by (\ref{impute.w})
for $X(t) = \hat{X}(T-t)$, $0\le t\le T$.
Then
\begin{math}
  \tilde{\Theta}_{y^*,T}(\hat{X}(T-\cdot)) =
  \Theta_{x,y^*,T}(\omega)
\end{math}
is distributed as $\mathbb{W}$ on $C([0,T],\mathbb{R}^n)$.
\end{proposition}

In Proposition~\ref{ligg.cons}
we have constructed the Markov dynamical system
$\Xi^*_t$ of Liggett dual to (\ref{c.sde}).
By replacing $\omega$ with $\tilde{\Theta}_{y^*,t}(X)$,
we obtain a dynamical system
\begin{equation}\label{psi.mds}
  \Psi^*_t((z^*,y^*),X) = \Xi^*_t((z^*,y^*),\tilde{\Theta}_{y^*,t}(X))
\end{equation}
from $\bar{D}^*\times C(\mathbb{R}_+,\mathbb{R}^n)$ to $\bar{D}^*$.

\begin{proposition}\label{s.ligg.cons}
Construct the intertwining dual $Q^*_t$ of (\ref{c.sde})
from $\Xi^*_t$ by Proposition~\ref{twin.prop}.
Then the dynamical system of (\ref{psi.mds})
is $\Lambda$-linked.
\end{proposition}

\begin{proof}
(a)
Let $(z^*,y^*)\in \bar{D}^*$ be fixed.
Consider a forward process  
$X(t)$, $0\le t\le T$, by~(\ref{sde.fwd}).
(i)
If $X(0)\not\in y^*$ then $\omega' = \tilde{\Theta}_{y^*,T}(X)$,
and therefore, $X(T)\not\in\Phi^{-1}_T(y^*,\omega(T-\cdot))$.
(ii)
If $X(0)\in y^*$ then $\omega^L = \tilde{\Theta}_{y^*,T}(X)$.
Provided the backward sample path
$\hat{\xi}(s) = \omega^L(T-s) - \omega^L(T)$,
$\hat{X}(s) = X(T-s)$, $0\le s\le T$, is a solution
to~(\ref{skd})--(\ref{skd.l}).
Thus, we obtain
$X(0) = \psi_{y^*,T}(X(T),\tilde{\Theta}_{y^*,T}(X)(T-\cdot))$.
Both (i) and (ii) together implies (\ref{s.ligg.g}).
(b)
By Proposition~\ref{y.hat.dist} we can show that
\begin{equation*}
  \mathbf{E}_{\mathbb{P}_x}\big[
    f(\Xi^*_T((z^*,y^*),\tilde{\Theta}_{y^*,T}(\hat{X}(T-\cdot))))
    \big]
\end{equation*}
generates the Liggett dual of Proposition~\ref{ligg.cons},
and that it is independent on the choice of initial state $x\in D$
for $\hat{X}$.
Having verified (c),
we can apply Proposition~\ref{s.twin.prop}.
\end{proof}

\section{Examples of intertwining dual}
\label{c.link}

In order to apply Proposition~\ref{s.twin.prop}
for the $\Lambda$-linked dynamical system $\Psi^*_t$ of
Proposition~\ref{s.ligg.cons},
the corresponding harmonic function of~(\ref{hgv})
must be finite and strictly positive on $D^*$.
Then we can form a $\Lambda$-linked semigroup by
\begin{equation*}
  V_tg((z^*,y^*),x)
  = \mathbf{E}_{\mathbb{W}}\big[
    g(\Xi^*_t((z^*,y^*),\tilde{\Theta}_{y^*,t}(\Phi_\cdot(x,\omega))),
    \Phi_t(x,\omega))
  \big]
\end{equation*}
over
\begin{equation*}
  E = \{((z^*,y^*),x)\in D^*\times D: x\in y^*\setminus z^*\}.
\end{equation*}
We can set $X^*(0) = x^*\in D^*$
and sample $X(0)$ randomly from
$\lambda(x^*,\cdot)$ for an initial state.
Then we can produce the $\Lambda$-linked coupling
$\mathbb{V}_{\lambda(x^*,\cdot)}$
by generating $X(t) = \Phi_t(X(0),\omega)$,
imputing $\tilde{\Theta}_{y^*,t}(X)$ from $X$,
and constructing
\begin{math}
  X^*(t) = \Xi^*_t(x^*,\tilde{\Theta}_{y^*,t}(X)).
\end{math}
The resulting bivariate process $(X^*(t),X(t))$
takes values on $E$,
and determines the intertwining dual $Q^*_t$ by
\begin{equation*}
  Q^*_t g(x^*) =
  \mathbf{E}_{\mathbb{V}_{\lambda(x^*,\cdot)}}[g(X^*(t))] .
\end{equation*}

\subsection{Entrance states for intertwining dual}
\label{e.link}

We continue Example~\ref{r1.ex.ligg},
and examine a construction of $\Lambda$-linked coupling.
In the setting of Proposition~\ref{twin.prop}
we find the finite and strictly positive harmonic function
\begin{math}
  h(z,y) = \displaystyle\int_z^y\nu_1(x)\,dx
\end{math}
for the Liggett dual of (\ref{r1.b}),
and obtain the differential operator
\begin{align*}
  \mathcal{B}^*f
  & = \frac{1}{h}\mathcal{B}[hf]
  \\
  & = \mathcal{B}f
  + \frac{1}{h}\left(
  \frac{\partial}{\partial y} - \frac{\partial}{\partial z}
  \right)\!h\times
  \left(
  \frac{\partial}{\partial y} - \frac{\partial}{\partial z}
  \right)\!f
\end{align*}
which corresponds to the intertwining dual $Q^*_t$
on the dual state space of (\ref{bm.d}).

In the next proposition
we use a Markov dynamical system
\begin{equation*}
  \Xi^*_t((z,y),\omega) = \big(
  \Phi^{-1}_t(z,\omega(t-\cdot)),
  \Phi^{-1}_t(y,-\omega(t-\cdot))
  \big) ,
\end{equation*}
and generate
$(X^*(t),X(t))$ starting from $(X^*(0),X(0)) = ((x,x),x)$.
It produces a probability measure
on $\bar{E} = \{((z,y),x):z\le x\le y\}$.

\begin{proposition}\label{ent.prop}
Let $(z_N)$ be an infinite sequence in $\mathbb{R}$ converging
to $x$ from below,
and let
$\mathbb{V}_{\lambda(x^*_N,\cdot)}$
be the $\Lambda$-linked coupling for each $x^*_N = (z_N,x)$.
Then (a)
$\mathbb{V}_{\lambda(x^*_N,\cdot)}$ converges weakly to
the distribution of $(X^*(t),X(t))$, $0\le t< \infty$,
and (b)
$X^*(t)$ is a diffusion process
associated with $\mathcal{B}^*$.
\end{proposition}

\begin{proof}
We can consider a Polish space $\mathbb{R}^{\mathbb{N}}$ of
infinite sequences, and construct a probability measure $\mathbb{B}$
on $\mathbb{R}^{\mathbb{N}}$ satisfying
$\mathbb{B}(B_N\in db_N) = \lambda(x^*_N,db_N)$.
Then we sample $((B_N),\omega)$ from
$\mathbb{B}\otimes\mathbb{W}$,
and generate $X(t) = \Phi_t(x,\omega)$
and $X_N(t) = \Phi_t(B_N,\omega)$.
Let $T > 0$ be fixed.
By applying the Gronwall's inequality to (\ref{c.ie})
we find that
$|X_N-X|_T = O(|z_N-x|)$,
and by Lemma~\ref{skd.ell.d} that
\begin{math}
  |\tilde{\Theta}_{x,T}(X_N) - \tilde{\Theta}_{x,T}(X)|_T
  = O(|z_N-x|) .
\end{math}

(a)
We can generate
\begin{math}
  X^*(t) = \Xi^*_t((x,x),\tilde{\Theta}_{x,t}(X))
\end{math}
and
\begin{math}
  X^*_N(t) = \Xi^*_t(x^*_N,\tilde{\Theta}_{x,t}(X_N)) ,
\end{math}
and verify
(cf. the proof of Lemma~\ref{bounded.gamma})
that
$\|X^*_N - X^*\|_T = O(|z_N-x|)$.
This implies
(via argument of Section~8 of~\cite{billingsley})
that
$\mathbb{V}_{\lambda(x^*_N,\cdot)}$ converges weakly to
the distribution of
$(X^*(t),X(t))$, $0\le t< \infty$.

(b)
Since $X^*_N(t)$ is a solution to the SDE for $\mathcal{B}^*$,
it satisfies for each $N$
the Dynkin's formula
\begin{equation}\label{r1.dyn}
  \mathbf{E}_{\mathbb{B}\otimes\mathbb{W}}\left[
    f(X^*_N(t))
    - \int_0^t\mathcal{B}^*f(X^*_N(v))dv
  \right]
  - f(x^*_N)
  = 0 .
\end{equation}
By letting $N\to\infty$
we find that (\ref{r1.dyn}) holds for $X^*(t)$,
completing the proof.
\end{proof}

By Proposition~\ref{ent.prop}
an intertwining dual process $X^*(t)$
is started at $X^*(0) = (x,x)$,
and viewed as SDE solution associated with
$\mathcal{B}^*$.
We set
\begin{math}
  m = \left(
  \frac{\partial}{\partial y} - \frac{\partial}{\partial z}
  \right)\!h ,
\end{math}
and by Ito chain rule we obtain
\begin{equation}\label{bessel.sde}
  d(h(X^*(t)))
  = \frac{m(X^*(t))^2}{h(X^*(t))}dt
  - m(X^*(t))dW(t) .
\end{equation}
Then we can construct a process
\begin{math}
  R(t) = \displaystyle\int_0^t m(X^*(v))^2 dv
\end{math}
and the time change $\tau_t = \inf\{v\ge0: R(v)>t\}$.
By using the time scale $\tau_t$
we obtain a three-dimensional Bessel process
$H(t) = h(X^*(\tau_t))$ satisfying
\begin{math}
  dH(t) = \dfrac{dt}{H(t)} - dW(t) .
\end{math}
starting from $H(0) = 0$,
and it never hits $0$ for $t > 0$.
In this sense
boundary points are viewed as
entrance states of the $\Lambda$-linked coupling.
This connection to three-dimensional Bessel process
was a critical observation
in the investigation of Miclo~\cite{miclo}.

\subsection{Further examples of entrance states}
\label{e.link.r2}

In this section we look at Example~\ref{r2.ex.ligg}
and~\ref{b.model.ligg}
for further example of entrance states.

\begin{example}\label{r2.ex.ent}
In Example~\ref{r2.ex.ligg} we set a dual state space
\begin{math}
  D^*_+ = \{(u,z,y)\in D^*:0< u_1< u_2\}
\end{math}
so that the harmonic function
\begin{equation*}
  h(u,z,y) = \sqrt{2\pi}
  \int_{(u_2z_1-u_1z_2)/\sqrt{2u_1u_2}}^{(u_2y_1-u_1y_2)/\sqrt{2u_1u_2}}
  e^{\eta^2}d\eta
\end{equation*}
is finite and strictly positive
for every $(u,z,y)\in D^*_+$.
Here the Markov dynamical system
\begin{equation*}
  \Xi^*_t((u,z,y),\omega) = \big(
  \Phi^{-1}_t(u,0),
  \Phi^{-1}_t(z,\omega(t-\cdot)),
  \Phi^{-1}_t(y,\omega'(t-\cdot))
  \big)
\end{equation*}
maps from $\bar{D}^*_+\times C([0,t],\mathbb{R}^n)$
to $\bar{D}^*_+$.
\end{example}

In Example~\ref{r2.ex.ent}
we obtain the intertwining dual operator
\begin{align*}
  \mathcal{B}^*f
  &
  = \mathcal{B}f
  + \frac{1}{h}\left(
  \frac{\partial}{\partial y_1} - \frac{\partial}{\partial z_1}
  \right)\!h\times
  \left(
  \frac{\partial}{\partial y_1} - \frac{\partial}{\partial z_1}
  \right)\!f
  \\ & \hspace{7.5ex}
  + \frac{1}{h}\left(
  \frac{\partial}{\partial y_2} + \frac{\partial}{\partial z_2}
  \right)\!h\times
  \left(
  \frac{\partial}{\partial y_2} + \frac{\partial}{\partial z_2}
  \right)\!f
\end{align*}
In the next proposition
we choose an initial point
$X^*(0) = (u,z,y)$ satisfying
$\langle[u_2,-u_1]^T,y-z\rangle = 0$
with $0< u_1 < u_2$,
and sample $X(0)$
from a pdf $\bar{\lambda}(\partial y^*,\cdot)$
proportional to $\nu(x_1,x_2)$
on the surface
\begin{math}
  \partial y^* =
  \{x: \langle[u_2,-u_1]^T,x-y\rangle = 0\} .
\end{math}
The resulting bivariate process
$(X^*(t),X(t))$ generates a probability measure on
\begin{math}
  \bar{E} = \{((u,z,y),x):
  \langle[u_2,-u_1]^T,y-x\rangle \ge 0,\,
  \langle[u_2,-u_1]^T,x-z\rangle \ge 0,\,
  0< u_1< u_2
  \}
\end{math}

\begin{proposition}\label{ent.prop.r2}
Let $(z_N)$ be an infinite sequence in $\mathbb{R}^2$ converging
to $z$ while $x^*_N = (u,z_N,y)\in D^*_+$,
and let
$\mathbb{V}_{\lambda(x^*_N,\cdot)}$
be the $\Lambda$-linked coupling.
Then Proposition~\ref{ent.prop}(a)--(b) holds.
\end{proposition}

\begin{proof}
We can consider a Polish space of infinite sequences
$(b_0,(b_N))$ with $b_0\in\partial y^*$ and $b_N\in\mathbb{R}^2$,
and construct a probability measure $\mathbb{B}$ satisfying
\begin{math}
  \mathbb{B}(B_0\in db_0) = \bar{\lambda}(\partial y^*,db_0)
\end{math}
and
\begin{math}
  \mathbb{B}(B_N\in db_N) = \lambda(x^*_N,db_N);
\end{math}
moreover, for each $k$ we can find
a sequence $c_{k,N}$ converging to zero
such that
\begin{math}
  \mathbb{B}\left(
  \|B_N-B_0\|\le c_{k,N},N\ge 1
  \right)\ge 1 - 2^{-k} .
\end{math}
Then we sample $((B_0,(B_N)),\omega)$ from
$\mathbb{B}\otimes\mathbb{W}$,
and generate
$X(t) = \Phi_t(B_0,\omega)$
and $X_N(t) = \Phi_t(B_N,\omega)$.
Having fixed $T> 0$ and
\begin{math}
  A_k = \{(b_0,(b_N)):
  \|b_N-b_0\|\le c_{k,N},\, N\ge 1\} ,
\end{math}
and similarly to the proof of Proposition~\ref{ent.prop}
we obtain
$\|X_N-X\|_T = O(c_{k,N})$,
and
\begin{math}
  \|\tilde{\Theta}_{(u,y),T}(X_N) - \tilde{\Theta}_{(u,y),T}(X)\|_T
  = O(c_{k,N})
\end{math}
if $(B_0,(B_N))\in A_k$.

(a)
We can verify
$\|X^*_N - X^*\|_T = O(c_{k,N})$ over $A_k$
for the construction of
\begin{math}
  X^*(t) = \Xi^*_t((u,z,y),\tilde{\Theta}_{(u,y),t}(X))
\end{math}
and
\begin{math}
  X^*_N(t) = \Xi^*_t(x^*_N,\tilde{\Theta}_{(u,y),t}(X_N)) ,
\end{math}
and obtain Proposition~\ref{ent.prop}(a).

(b)
Assuming that $f$ and $\mathcal{B}^*f$ are bounded,
for each $\varepsilon> 0$ we can find sufficiently large $k$ so that
\begin{equation*}
  \left|\mathbf{E}_{\mathbb{B}\otimes\mathbb{W}}\left[
    \left(f(X^*_N(t))
    - \int_0^t\mathcal{B}^*f(X^*_N(v))dv
    \right) I_{A_k}(B_0,(B_N))
    \right]
  - f(x^*_N)
  \right| < \varepsilon
\end{equation*}
By letting $N\to\infty$, we can verify
Proposition~\ref{ent.prop}(b).
\end{proof}

By setting
\begin{equation*}
  m = \sqrt{
    \left[\left(
      \frac{\partial}{\partial y_1} - \frac{\partial}{\partial z_1}
      \right)\!h\right]^2
    + \left[\left(
      \frac{\partial}{\partial y_2} + \frac{\partial}{\partial z_2}
      \right)\!h\right]^2
  }
\end{equation*}
we can obtain (\ref{bessel.sde}) for an intertwining dual process
$X^*(t)$ of Proposition~\ref{ent.prop.r2},
and show that $h(X^*(t))$ never hits $0$ for $t > 0$,
as discussed in Section~\ref{e.link}.

\begin{example}\label{b.model.ent}
We continue to choose $\theta=\pi/2$ from
Example~\ref{b.model.ligg}.
Then we can introduce a harmonic function
$h(z,y) = \langle d,y-z\rangle$ for the Liggett dual $\mathcal{B}$,
and obtain the intertwining dual
\begin{equation*}
  \mathcal{B}^*f
  = \mathcal{B}f
  + \frac{2d_1}{h(z,y)}\left(
  \frac{\partial}{\partial y_1} - \frac{\partial}{\partial z_1}
  \right)\!f .
\end{equation*}
Similarly to Proposition~\ref{ent.prop.r2} we can set an initial
point $X^*(0) = (z,y)$ on the boundary
(i.e., $\langle d,y-z\rangle = 0$),
and sample $X(0)\in(H+y)$ randomly
so that
$X(0)-\langle d,X(0)\rangle d$ is distributed as
the pdf $\nu_H(x)$ proportional to $\nu(x)$ on $H$.
The resulting probability measure $\mathbb{V}_{(z,y)}$
of bivariate process $(X^*(t),X(t))$ is the limiting
distribution of the
$\Lambda$-linked coupling $\mathbb{V}_{\lambda(x^*_N,\cdot)}$
if a sequence $x^*_N = (z_N,y)\in D^*$ converges to $(z,y)$,
for which Proposition~\ref{ent.prop}(a)--(b) is similarly
verified.
We can also obtain (\ref{bessel.sde}) with $m(z,y)\equiv 2d_1$.
\end{example}

\begin{remark}\label{b.model.rem}
(a)
As in the discussion of Example~\ref{b.model.ligg},
a hyperplane $H_1$ not parallel to $H$ can be used for an initial
state.
Proposition~\ref{s.ligg.cons} is applicable,
but a $\Lambda$-linked dynamical system $\Psi^*_t$
becomes intractable with the
choice of $\mathbb{F}_2 = \mathbb{F}_0\cup\{\varnothing\}$.
(b)
In Example~\ref{b.model} the input vectors
$a^{(1)},\ldots,a^{(N)}$ can span the entire space $\mathbb{R}^n$,
and $\nu(x)$ can be viewed as
a posterior density on $\mathbb{R}^n$
up to the normalizing constant
$z_{\nu} = \int_{\mathbb{R}^n}\nu(x)dx$;
the finiteness of $z_{\nu}$
can be checked similarly to (\ref{nu.finite}).
However, it is no longer guaranteed that
a sample path $Y^*(t) = \Phi^{-1}_t(Y^*(0),\omega(t-\cdot))$
remains in a subclass $\mathbb{F}_1$
of hypographic closed sets.
\end{remark}

In spite of Remark~\ref{b.model.rem}
we conclude this study by suggesting
a scheme for Monte Carlo simulation.
In light of (\ref{vee.cond}) it can be designed to sample
$X(T)\in R$ from the pdf proportional to $\nu(x)$
for some fixed region $R$ of $\mathbb{R}^n$.
Set an initial state $\partial Z^*(0) = \partial Y^*(0) = H_1$,
and sample $X(0)\in H_1$ from the pdf proportional to $\nu(x)$.
Construct a sample path $(X^*(t),X(t))$ of $\Lambda$-linked coupling,
and stop it at time $T$ when $R\subseteq Y^*(T)\setminus Z^*(T)$.
Accept a sample $X(T)$ from the pdf of interest if $X(T)\in R$.

\bibliographystyle{plain}
\bibliography{intertwin}

\end{document}